\documentclass[11pt]{article}
\usepackage[margin=1in]{geometry}
\usepackage{amsmath,amssymb,amsfonts}
\usepackage{amsthm}
\usepackage{graphicx}
\usepackage{bm}
\usepackage{booktabs}
\usepackage{mathtools}
\usepackage{hyperref}
\usepackage{authblk}
\usepackage{algorithm}
\usepackage[title]{appendix}
\usepackage[noend]{algpseudocode}  
\usepackage{xcolor}
\newtheorem{theorem}{Theorem}[section]
\newtheorem{corollary}{Corollary}[theorem]
\newtheorem{lemma}[theorem]{Lemma}
\newtheorem{proposition}[theorem]{Proposition}
\newtheorem{definition}[theorem]{Definition}

\newtheorem{remark}[subsection]{Remark}
\hypersetup{colorlinks=true,linkcolor=blue,citecolor=blue,urlcolor=blue}
\usepackage{enumitem}
\newcommand{\Nn}{\mathcal{N}}
\newcommand{\RR}{\mathbb{R}}
\newcommand{\Hh}{\mathbb{H}^{2\alpha}}
\newcommand{\flp}{(-\Delta)^\alpha}
\newcommand{\flpg}{(-\Delta_g)^\alpha}

\newcommand{\dom}{\text{Dom}}
\title{Wellposedness and dynamics of two types of reaction--nonlocal diffusion systems under the inhomogeneous spectral fractional Laplacian}
\author{Pu Yuan, Paul A. Zegeling}
\date{\today}

\begin{document}
\maketitle

\begin{abstract}
Reaction–nonlocal diffusion equations model nonlocal transport and anomalous diffusion by replacing the Laplacian with a fractional power, capturing diffusion mechanisms beyond Brownian motion. We primarily study the semilinear problem 
\[
\partial_t u + \epsilon^2\flpg u = \Nn(u)
\]
allowing constant inhomogeneous Dirichlet boundary condition $u|_{\partial\Omega}=g$. To handle the boundary constraint, we use a harmonic lifting to reformulate the problem as an equivalent homogeneous system with a shifted nonlinearity. Working in \(C_0(\Omega)\), analytic contraction semigroup theory yields the Duhamel formula and quantitative smoothing, implying local well-posedness for locally Lipschitz reactions and a blow-up alternative. The semigroup viewpoint also provides $L^\infty$-contractivity and positivity preservation, which drive pointwise maximum principles and stability bounds. Furthermore, we analyze two prototypes. For the bistable RNDE, we derive an energy dissipation identity and, using a fractional weak maximum principle, obtain an invariant-range property that confines solutions between the two stable steady states. For the nonlocal Gray--Scott system with possibly different fractional diffusion orders, we prove that solutions preserve positivity. Moreover, we identify an explicit \(L^\infty\) invariant set ensuring global boundedness, and derive an eigenfunction-weighted interior \(L^2\) bound. Finally, we perform numerical simulations using a sine pseudospectral discretization and ETDRK4 time-stepping, which illustrate the impact of fractional orders on pattern formation, consistently with our analytical results.

\vspace{1em}

\vspace{1em}
\end{abstract}

\section{Introduction}

Reaction--nonlocal diffusion equations (RNDEs) generalize classical reaction--diffusion models by replacing the Laplacian with a fractional power $(-\Delta)^\alpha$ with $\alpha \in (0,1)$, effectively capturing nonlocal transport and anomalous diffusion across physics and biology.
The Dirichlet, or integral, fractional Laplacian is defined through a singular integral after zero extension outside $\Omega$ and corresponds to a symmetric $2\alpha$-stable process killed when it exits the domain \cite{ros2016nonlocal,frac_laplacian_a0,frac_laplacian,bucur2016nonlocal}. The spectral fractional Laplacian, used in this paper, is instead the fractional power of the Dirichlet Laplacian and is defined through the eigenpairs $\{(\lambda_k,\varphi_k)\}$ of $-\Delta$ on $\Omega$ \cite{frac_laplacian,antil2015fem}. It is therefore tied to the bounded-domain heat semigroup and to the geometry encoded by the Dirichlet eigenfunctions, this makes it natural for models in which diffusion is constrained by the boundary of the physical domain and for sine/eigenfunction-based discretizations \cite{cusimano2018discretizations,servadei2014spectrum}. 

This distinction becomes particularly important for semilinear RNDEs, where the boundary condition, the nonlocal diffusion operator, and the nonlinear reaction term must be treated in a compatible functional framework.
Well-posedness, qualitative dynamics, and numerical approximation for fractional semilinear equations have been studied from several viewpoints, including semilinear fractional elliptic problems \cite{antil2015fem}, fractional Allen--Cahn dynamics related to phase separation and bistable fronts \cite{sohaib2024space,chan2017traveling}, fractional Gray--Scott-type models \cite{wang2019fractional,yuan2024adaptive}, and nonlocal kernel-based reaction--diffusion systems \cite{laurenccot2023nonlocal,jian2020fastcompact}. These works provide important analytical and computational foundations, but they usually focus either on homogeneous spectral data, integral/Riesz-type operators, or numerical dynamics without a systematic treatment of inhomogeneous Dirichlet problems driven by the spectral fractional Laplacian. The purpose of the present work is to combine harmonic lifting, analytic semigroup theory in $C_0(\Omega)$, maximum-principle arguments, and spectral time-stepping to obtain a coherent framework for boundary-driven RNDEs involving the inhomogeneous spectral fractional Laplacian.

We consider the semilinear problem with a time-independent inhomogeneous Dirichlet boundary condition
\begin{equation}\label{inhomo}
\left\{
\begin{aligned}
&\partial_t u + \epsilon^2 \flpg u = \Nn(u), \qquad x\in\Omega,\ t>0,\\
&u|_{\partial\Omega}=g,
\end{aligned}
\right.
\end{equation}
where $\epsilon>0$ is the diffusion coefficient, $\Nn$ denotes the reaction operator, and $g$ is the prescribed boundary value. Such conditions are essential in many models to represent external feeds or imposed constraints. However, for spectral fractional Laplacians, inhomogeneous boundary conditions cannot be treated by simply repeating the standard definition of fractional powers, as the class of smooth functions satisfying a fixed nonzero boundary condition does not form a linear space \cite{frac_laplacian,cusimano2018discretizations,Spectral_frac_def}. To address this, we introduce the inhomogeneous spectral fractional Laplacian $\flpg$, which preserves the spectral structure while incorporating the boundary value.

Following the harmonic lifting approach \cite{frac_laplacian}, we reduce \eqref{inhomo} to a homogeneous problem. Let $G$ be the harmonic lifting of $g$, and set $w=u-G$. Then $w$ satisfies a homogeneous Dirichlet condition and solves
\[
\partial_t w + \epsilon^2\flp w = \Nn(w+G), \qquad w|_{\partial\Omega}=0.
\]
The resulting homogeneous formulation has a linear part governed by an analytic semigroup. Specifically, the operator $-\epsilon^2\flp$ generates an analytic contraction semigroup $\{e^{-\epsilon^2\flp t}\}_{t \ge 0}$ on the Banach space $C_0(\Omega)$. The variation-of-constants formula then provides the essential estimates for establishing well-posedness, time-regularity, and the long-time dynamics of solutions. However, the reduction also changes the nonlinearity through $\Nn(\cdot)\mapsto \Nn(\cdot+G)$, which may break sign conditions, monotonicity, and dissipative structures required for global energy estimates. This issue is particularly relevant for coupled systems and is addressed in Section~5.

The first main contribution is the well-posedness theory for the homogeneous fractional reaction--diffusion equation \eqref{ori_eq} in the Banach space $C_0(\Omega)$. Since the spectral fractional Laplacian is $m$-accretive, the operator $-\epsilon^2(-\Delta)^\alpha$ generates a contractive analytic semigroup $Q(t)=e^{-\epsilon^2(-\Delta)^\alpha t}$. Assuming that $\mathcal N:C_0(\Omega)\to C_0(\Omega)$ is locally Lipschitz on bounded sets, the variation-of-constants formula yields a unique mild solution on a maximal time interval,
\[
u(t)=Q(t)u_0+\int_0^t Q(t-s)\,\Nn(u(s))\,ds .
\]
We then recall an abstract blow-up alternative for the maximal existence time. Moreover, under the one-sided quadratic bound $u\,\mathcal N(u)\le C u^2$ for sufficiently large $|u|$, an $L^2$ a priori estimate gives global existence. Finally, standard analytic-semigroup estimates imply positive-time regularity: $u(t)\in\mathbb H^{2\alpha}(\Omega)$ for $t>0$ and $u\in C^1((0,T],C_0(\Omega))$. This provides the abstract well-posedness framework used in the rest of the paper.

The second main contribution applies this framework to two classical models: a bistable RNDE and a coupled fractional reaction--diffusion system. For the bistable equation, after harmonic lifting, we define an energy functional and prove an energy dissipation identity, which implies monotone decay of the energy and yields a coercivity estimate. Consequently, weak solutions are uniformly bounded in the fractional energy space $\mathbb{H}^{\alpha}(\Omega)$. Combining the compact embedding $\mathbb{H}^{\alpha}(\Omega)\hookrightarrow L^2(\Omega)$ with the dissipation estimate, we prove that the $\omega$-limit set is nonempty and that every $\omega$-limit point is a stationary solution in the weak sense. A fractional weak maximum principle further gives an invariant-range property: solutions initially lying between the stable steady states remain in this range for all time. The result is then extended to the inhomogeneous operator $\flpg$. Numerically, the fractional Nagumo equation is used as a representative example, sine pseudospectral discretization and an ETDRK4 time integrator are employed to illustrate the effect of the fractional order $\alpha$ on spreading and to confirm the boundedness predicted by the analysis.

For the coupled case, we consider a nonlocal Gray--Scott model in which the two components $(u,v)$ may diffuse with different fractional orders $\alpha$ and $\beta$. Since the reactant $u$ is maintained at a fixed level on $\partial\Omega$, the boundary data are inhomogeneous. As in the bistable case, a lifting approach converts the original system into an equivalent homogeneous Dirichlet system. For this homogeneous formulation, we prove a sharp invariant structure: positivity is preserved, with $v(\cdot,t)\ge0$ and $u(\cdot,t)\ge -1$, and an explicit $L^\infty$ invariant set is identified. In particular, $\|u(\cdot,t)\|_\infty\le1$ and $\|v(\cdot,t)\|_\infty\le \frac{F+\kappa}{2-\|u_0\|_\infty}$ for all $t\ge0$, so trajectories remain bounded and the solution extends globally in time. We also derive an eigenfunction-weighted interior $L^2$ estimate for $v$, which controls the autocatalyst away from the boundary through the first Dirichlet eigenpair of $(-\Delta)^\alpha$. The same numerical method is then used to simulate distinct pattern formations generated by different choices of $(F,\kappa)$ and fractional orders.

The remainder of the paper is organized as follows. Section~2 recalls the spectral and Dirichlet fractional Laplacians, the relevant Sobolev scales, and the lifting construction for inhomogeneous spectral boundary conditions. Section~3 proves the semigroup well-posedness results. Sections~4 and~5 treat the bistable equation and the fractional Gray--Scott system, respectively, including numerical experiments. Appendix~A records the DST-I/ETDRK4 discretization details.

\section{Preliminaries}
In this section, we introduce the spectral fractional Laplacian and its corresponding Sobolev spaces. When a bounded domain is considered, the spectral fractional Laplacian is significantly different from the Dirichlet fractional Laplacian. In the following, we will introduce and compare the properties of these two operators from various aspects. Let $\Omega\subset\RR^n$ $(n\in\mathbb{N}^+)$ be a bounded open Lipschitz domain with boundary $\partial\Omega$ and complement $\Omega_c=\RR^n\setminus\Omega$. We write $(\cdot,\cdot)$ for the $L^2(\Omega)$ inner product, $\|\cdot\|_2$ for the $L^2(\Omega)$ norm, and $\|\cdot\|_\infty$ for the $L^\infty(\Omega)$ norm. Let $\{(\lambda_k,\varphi_k)\}_{k\ge1}$ be the Dirichlet eigenpairs of the classical Laplacian $-\Delta$ on $\Omega$, ordered so that $0<\lambda_1\le\lambda_2\le\cdots<\infty$, with $\{\varphi_k\}$ an orthonormal basis of $L^2(\Omega)$.

\subsection{Spectral fractional Laplacian}

\begin{definition}\label{spec_def}
The spectral fractional Laplacian of $u$ in $\Omega$ is given by
\[
(-\Delta)^\alpha u(x)=\sum_{k=1}^\infty \lambda_k^\alpha\, u_k\,\varphi_k(x),
\qquad u_k=\int_\Omega u\,\varphi_k\,dx,
\]
for $\alpha\in(0,1)$. 
\end{definition}

This operator is exactly the fractional power of the Dirichlet Laplacian with zero boundary conditions. From this construction, $(-\Delta)^\alpha$ is positive, self-adjoint, and densely defined on $L^2(\Omega)$, see \cite{balakrishnan1960fractional}.

On a bounded domain $\Omega$, the following semigroup representation is equivalent to Definition~\ref{spec_def} (see \cite{cusimano2018discretizations}) and we introduce it because it directly links $(-\Delta)^\alpha$ to the heat flow, which is positivity preserving and $L^\infty$-contractive. The equivalence is immediate from the spectral theorem: expand $u=\sum_{k}u_k\varphi_k$, note that $e^{t\Delta}\varphi_k=e^{-\lambda_k t}\varphi_k$, and compute the integral to get $(-\Delta)^\alpha u=\sum_k \lambda_k^\alpha u_k\varphi_k$:
\begin{definition}
The spectral fractional Laplacian of $u$ in $\Omega$ admits the semigroup representation
\[
(-\Delta)^\alpha u(x)=\frac{1}{\Gamma(-\alpha)} \int_0^{\infty}\big(e^{t \Delta} u(x)-u(x)\big)\,\frac{dt}{t^{1+\alpha}},
\]
where $e^{t \Delta} u$ solves the classical heat equation on $\Omega$ with homogeneous Dirichlet boundary condition.
\end{definition}

\subsection{Fractional Sobolev spaces}

The fractional Sobolev space $H^\alpha(\Omega)$ of order $\alpha\in(0,1)$ is
\begin{equation}\label{eq:Gagliardo}
H^\alpha(\Omega)
=\Bigl\{u\in L^2(\Omega):\ [u]_{H^{\alpha}(\Omega)}^{2}:=\iint_{\Omega \times \Omega}
\frac{|u(x)-u(y)|^{2}}{|x-y|^{n+2\alpha}}\,dx\,dy< \infty\Bigr\},
\end{equation}
equipped with the norm $\|u\|_{H^\alpha(\Omega)}=\bigl(\|u\|_2^2+[u]_{H^\alpha(\Omega)}^2\bigr)^{1/2}$, which makes it a Hilbert space.

\begin{definition}\cite{frac_laplacian,Spectral_frac_def,Spectral_frac_def2}
Let $\alpha\in(0,1)$. Define the spectral fractional power space
\[
\mathbb{H}^\alpha(\Omega) = \left\{u=\sum_{k=1}^\infty u_k\varphi_k\in L^2(\Omega):\ \sum_{k=1}^\infty \lambda_k^{\alpha}|u_k|^2<\infty\right\},
\]
with norm $\|u\|^2_{\mathbb{H}^{\alpha}(\Omega)}=\|(-\Delta)^{\frac{\alpha}{2}}u\|_{2}^{2}=\sum_{k=1}^\infty \lambda_k^{\alpha}|u_k|^2$.
Moreover,
\[
\mathbb{H}^\alpha(\Omega)=
\begin{cases}
H^\alpha(\Omega), & \alpha\in\bigl(0,\tfrac{1}{2}\bigr),\\[3pt]
H^{\frac{1}{2}}_{00}(\Omega)=\Bigl\{u\in H^{\frac{1}{2}}(\Omega):\displaystyle\int_\Omega\frac{u(s)^2}{\mathrm{dist}(s,\partial\Omega)}\,ds<\infty\Bigr\}, & \alpha=\tfrac{1}{2},\\[5pt]
H_{0}^{\alpha}(\Omega)=\{u\in H^{\alpha}(\Omega): u|_{\partial\Omega}=0\}, & \alpha\in \bigl(\tfrac{1}{2},1\bigr).
\end{cases}
\]
\end{definition}

\subsection{Inhomogeneous spectral fractional Laplacian}

While homogeneous boundary conditions are often adopted for simplicity and are convenient for analysis, they are not always appropriate in applications, where external forcing, imposed profiles, or constraints naturally lead to nonzero boundary conditions. For this reason, it is necessary to discuss the inhomogeneous case and to clarify which parts of the homogeneous theory require modification. In particular, although spectral methods are well suited to homogeneous boundary conditions, the construction of the spectral fractional powers in Definition \ref{spec_def} cannot be directly extended to the case of nonzero boundary conditions. The reason is subsets of $C^\infty(\Omega)$ that satisfy a fixed nonzero boundary condition are no longer linear spaces, which prohibits the use of the spectral theorem. A standard way around this difficulty is to subtract a harmonic lifting of the boundary data and then apply the homogeneous spectral operator to the zero-trace part \cite{frac_laplacian,Spectral_frac_def}.

Let $G$ denote the harmonic lifting of the boundary datum $g$, namely
\[
-\Delta G=0 \quad \text{in } \Omega,
\qquad
G|_{\partial\Omega}=g.
\]
For the constant boundary value considered in this paper, we identify $g$ with the constant function on $\Omega$, so that the harmonic lift is simply $G\equiv g$. Following the lifting construction in \cite{Spectral_frac_def}, the inhomogeneous spectral fractional Laplacian is defined by
\[
\flpg u := \flp (u-G),
\qquad u-G\in \Hh(\Omega).
\]
Equivalently, for sufficiently regular $u$, it admits the spectral representation
\begin{equation}
    \flpg u(x)
=
\sum_{k=1}^{\infty}
\left(
\lambda_k^\alpha u_k
+
\lambda_k^{\alpha-1} u_{\partial\Omega,k}
\right)
\varphi_k(x),
\end{equation}
where
\(
u_k=\int_\Omega u\varphi_k\,dx\) and   
\(u_{\partial\Omega,k}
=
\int_{\partial\Omega}
g\,\frac{\partial\varphi_k}{\partial n}\,d\sigma,
\)
with $n$ denoting the outward unit normal. Indeed, if
\[
G_k=\int_\Omega G\varphi_k\,dx,
\]
then Green's identity gives
\[
G_k
=
-\lambda_k^{-1}
\int_{\partial\Omega}
g\,\frac{\partial\varphi_k}{\partial n}\,d\sigma,
\]
and hence
\[
\lambda_k^\alpha (u_k-G_k)
=
\lambda_k^\alpha u_k
+
\lambda_k^{\alpha-1}u_{\partial\Omega,k}.
\]
Thus, in the constant-boundary case used below,
\begin{equation}
    \flpg u = \flp(u-g).
    \label{theorem44_18}
\end{equation}
This identity is the form used in the analysis.

Cusimano et al. \cite{cusimano2018discretizations} give an equivalent heat-semigroup definition, denoted by $\mathcal{L}^\alpha_{\Omega,g}$. If $e^{t\Delta_g}u$ denotes the solution at time $t$ of the heat equation with boundary value $g$ and initial datum $u$, then
\[
\mathcal{L}^\alpha_{\Omega,g}u
=
-\frac{1}{\Gamma(-\alpha)}
\int_0^\infty
\bigl(e^{t\Delta_g}u-u\bigr)
\frac{dt}{t^{1+\alpha}} .
\]
For constant $g$, this agrees with the lifting formula above, since
\[
e^{t\Delta_g}u
=
g+e^{t\Delta_0}(u-g),
\]
and therefore
\[
\mathcal{L}^\alpha_{\Omega,g}u
=
\flp(u-g).
\]

\subsection{Comparison to the Dirichlet fractional Laplacian}

\begin{definition}
For $\alpha\in(0,1)$, the Dirichlet fractional Laplacian $\flp_D$ on $\Omega$ is
\[
\flp_D u(x)= C(n,\alpha)\,\mathrm{P.V.}\int_{\RR^n}\frac{u(x)-u(y)}{|x-y|^{n+2\alpha}}\,dy,\qquad x\in\Omega,
\]
where $u$ is extended by zero on $\Omega_c$. The natural energy space is
\[
\widetilde H^\alpha(\Omega)=\bigl\{u\in H^\alpha(\RR^n):\ u=0\ \text{a.e. in }\RR^n\setminus\Omega\bigr\}.
\]
The operator maps $u\in\widetilde H^\alpha(\Omega)$ to $(\widetilde H^\alpha(\Omega))^*$ in the weak sense.
\end{definition}

The operator $\flp_D$ is the generator of the symmetric $2\alpha$-stable process killed upon exiting $\Omega$, hence it is also referred to as the restricted fractional Laplacian by Lischke~\cite{frac_laplacian} and Vázquez~\cite{vazquez2014recent}, and it differs from the spectral one unless $\alpha=1$.

The domains of $\flp$ and $\flp_D$ as operators into $L^2(\Omega)$ are, respectively,
\[
\mathrm{Dom}(\flp)=\mathbb{H}^{2\alpha}(\Omega)
\quad\text{and}\quad
\mathrm{Dom}(\flp_D)=\widetilde H^{2\alpha}(\Omega),
\]
whenever these sets are interpreted in the graph-norm sense, see \cite{rel_spec_dir}. In particular, for the spectral operator the domain can be identified more concretely (see \cite{spaces_dir_spec_1,spaces_dir_spec_2,spaces_dir_spec_3}) as
\begin{equation}\label{eq:spectral-domain}
\mathrm{Dom}(\flp)=
\begin{cases}
H^{2\alpha}(\Omega), & \alpha\in(0,\tfrac{1}{4}),\\[2pt]
H^{1/2}_{00}(\Omega), & \alpha=\tfrac{1}{4},\\[2pt]
H_{0}^{2\alpha}(\Omega), & \alpha\in(\tfrac{1}{4},\tfrac{1}{2}],\\[2pt]
H^{2\alpha}(\Omega)\cap H^1_0(\Omega), & \alpha\in(\tfrac{1}{2},1).
\end{cases}
\end{equation}

The difference between $\mathrm{Dom}(\flp)$ and $\mathrm{Dom}(\flp_D)$ reflects the different boundary interactions: $\flp$ imposes a spectral (local) Dirichlet trace encoded via $H_0^\theta$ when $2\alpha>\tfrac12$, whereas $\flp_D$ enforces the nonlocal zero extension through $\widetilde H^{2\alpha}(\Omega)$.

\begin{lemma}\cite{stinga2019user}
For any $\alpha_1,\alpha_2>0$, we have $(-\Delta)^{\alpha_1}\circ(-\Delta)^{\alpha_2} = (-\Delta)^{\alpha_1+\alpha_2}$.
\end{lemma}
A similar semigroup law holds for the Dirichlet fractional Laplacian on $L^2(\RR^n)$ under the Fourier definition, see Di~Nezza--Palatucci--Valdinoci~\cite{hitchhiker}. However, it is not generally holds on a bounded domain since $\flp_D u\ne 0$ in $\Omega\backslash\RR^n$ even if $u \equiv0$. Based on this semigroup property, we obtain the following lemma.

\begin{lemma}\label{lemma_spectrum_compare}
For all $u, v \in \mathbb{H}^\alpha(\Omega)$, 
\[
\big((-\Delta)^\alpha u, v\big)=\big(u,(-\Delta)^\alpha v\big)=\big((-\Delta)^{\frac{\alpha}{2}} u,\,(-\Delta)^{\frac{\alpha}{2}} v\big).
\]
Furthermore, if $0<\beta\le\alpha$, the spectral inequality
\[
\|(-\Delta)^{\beta/2}v\|_2\le \lambda_1^{\frac{\beta-\alpha}{2}}\|(-\Delta)^{\alpha/2}v\|_2
\]
holds.
\end{lemma}

The identities and the inequality follow termwise from the eigenfunction expansion and from the monotonicity of $\lambda_k^{\alpha}$ with respect to $\alpha$. It implies, for instance, that $(-\Delta)^{\theta}$ is smoothing of order $\beta-\alpha$ on $\mathbb{H}^\beta(\Omega)$. As a direct consequence of Lemma~\ref{lemma_spectrum_compare} we obtain:
\begin{corollary}\label{lemma_spectrum_split}
For all $u, v \in \mathbb{H}^\alpha(\Omega)$, 
\[
\big((-\Delta)^\alpha u, u\big)=\|(-\Delta)^{\frac{\alpha}{2}} u\|_2^2,\qquad 
\lambda_1^\alpha\|u\|_2^2\le\big((-\Delta)^{\alpha} u,u\big)\le\lambda_1^{-\alpha}\|u\|^2_{\mathbb{H}^\alpha(\Omega)}.
\]
\end{corollary}

\begin{proposition}\cite{dirichlet_spectral}
Let $\lambda_k$ and $\zeta_k$ be the $k$-th eigenvalues of the standard Laplacian $-\Delta$ and the Dirichlet-0 fractional Laplacian $\flp_D$, respectively. For any convex $\Omega\subset\RR^n$ and $\alpha\in(0,1)$,
\[
\frac{1}{2}\lambda_k^\alpha\ \le\ \zeta_k\ \le\ \lambda_k^\alpha.\label{compare_2_Laplacians}
\]
\end{proposition}

Thus, on convex domains the Dirichlet  spectrum is comparable to the spectral one, uniformly in $k$. This is frequently used to transfer a priori estimates between the two operators.

\section{Well-posedness for RNDE \eqref{inhomo}}

Using the harmonic lifting construction introduced above, we reduce \eqref{inhomo} to a homogeneous problem. Let $G$ be the harmonic lifting of $g$, for the constant boundary value considered in this paper, $G\equiv g$. With $w=u-G$, the equation becomes
\[
\partial_t w + \epsilon^2\flp w = \widetilde{\Nn}(w), \qquad w|_{\partial\Omega}=0,
\]
where $\widetilde{\Nn}(w):=\Nn(w+G)$ collecting the shifted nonlinearity. For convenience, we retain the notation $u$ for the shifted variable and formulate the reduced system as:
\begin{equation}
\left\{\begin{aligned}
& \partial_t u + \epsilon^2\flp u = \Nn(u), \qquad x\in\Omega,\ t>0,\\
& u|_{\partial\Omega}=0.
\end{aligned}\right.\label{ori_eq}
\end{equation}
Since $(-\Delta)^\alpha$ is $m$-accretive and densely defined on $L^2(\Omega)$, the operator $-(-\Delta)^\alpha$ generates a contraction semigroup $\{Q(t)\}_{t\ge0}$. Specifically, for any $t\ge0$, we have$$\|Q(t)\|_{L^2(\Omega)\to L^2(\Omega)}^2 = \sup_{u\ne 0}\frac{\sum_{k=1}^{\infty} \lambda_k^{2\alpha} e^{-2\epsilon^2 t\lambda_k^{\alpha}} |u_k|^{2}}{\sum_{k=1}^{\infty} \lambda_k^{2\alpha} |u_k|^{2}}=\sup_{k\ge 1} e^{-2\epsilon^2 t\lambda_k^{\alpha}}=e^{-2\epsilon^2 t\lambda_1^{\alpha}}\le 1.$$Moreover, $Q(t)$ is an analytic semigroup \cite{sectorial_and_analytic,haase_ana_semigroups}. More generally, on any Banach space $Y$, $-(-\Delta)^\alpha$ serves as the infinitesimal generator of a positive contraction semigroup (see, e.g., \cite{balakrishnan1960fractional}).
Suppose $\Nn$ is Lipschitz continuous on bounded subsets of a Banach space $Y$, i.e., for all $R>0$, there exists a constant $L_R$ such
 that
 \begin{equation}
     \|\Nn(u)-\Nn(v)\|_Y\le L_R \|u-v\|_Y,\  u,v \in B_R,
 \end{equation}
 where $B_R$ is the ball defined by $B_R:=\{u\in Y:\|u\|_{Y}\le R\}$ of center 0 and of radius $R$. Here we consider the case that $\Nn$ is a superposition operator, that is, there exists a locally Lipschitz
continuous function $f\in C(\RR^n,\RR^n)$ such that $\Nn(u)(x) = f(u(x))$ with $f(0) = 0$ for $u\in Y$ and  $x\in\Omega$.

 With $-\epsilon^2\flp$ and $Q(t)=e^{-\epsilon^2 \flp t}$ defined, we then can construct the solution $u(t)$ of RNDE \eqref{ori_eq} by using the \textit{variation-of-constants} (Duhamel) formula. Then, we have the following theorem:

\begin{theorem}\label{local_mild_exist}
    Suppose $u_0\in C_0(\Omega):=\{u_0\in C(\bar\Omega):u_0 = 0\text{ on }\partial \Omega\}$ and $\Nn: C_0(\Omega)\to C_0(\Omega)$ is locally Lipschitz on bounded sets of $C_0(\Omega)$. Then, there exists a unique $u\in C([0,T],C_0(\Omega))$ given by
    \begin{equation}
        u(t) = Q(t) u_0 + \int^t_0 Q(t-s)\Nn(u(s))ds,\ \ t\in[0,T],\label{mild_sol}
    \end{equation}
    where $Q(t)$ is an analytic contraction semigroup generated by $-\epsilon^2\flp$.
    \begin{proof}
        Suppose $M\ge \|u_0\|_\infty$. Let $R=2M+\|\Nn(0)\|_\infty$ and $L_R$ be a Lipschitz constant of $\Nn$ on the closed ball $B_R:=\{v\in C_0(\Omega):\|v\|_\infty\le R\}$. Choose $T_M>0$ such that 
        \begin{equation} 
        T_M = \frac{1}{2L_R+2}. \label{TR_condition} \end{equation}
        Define 
        \[Y = \{u\in C([0,T_M],C_0(\Omega)):\|u(t)\|_\infty\le R,\ \forall t\in[0,T_M] \}\]
        equipped with the distance generated by the norm of $C([0,T_M],C_0(\Omega))$, that is,
        \[d(u,v) = \max_{t\in[0,T_M]}\|u(t)-v(t)\|_\infty,\ \text{ for }u,v\in Y.\]
        Then, $(Y,d)$ is a complete metric space, since $C([0,T_M],C_0(\Omega))$ is a Banach space. 
        Since $Q(t)$, generated by $-\epsilon^2\flp$, is a contraction semigroup and $\Nn$ is locally Lipschitz, the following fixed-point map is well-defined on $Y$ for sufficiently small $T_M$.
        For all $u\in Y$, we define $(\Phi u)\in C([0,T_M],C_0(\Omega))$  by
        \begin{equation}
            (\Phi u)(t) = Q(t)u_0 + \int^t_0 Q(t-s) \Nn(u(s))ds,\ \forall t\in[0, T_M].
        \end{equation}
        Then we find
        \[\begin{aligned}
            \|(\Phi u)(t)\|_\infty&\leq \|Q(t)u_0\|_\infty+\int^t_0\|Q(t-s)\Nn(u(s))\|_\infty ds\leq \|u_0\|_\infty + \int^t_0\|\Nn(u(s))\|_\infty ds\\
            &\le \|u_0\|_\infty + \int^t_0 \|\Nn(0)\|_\infty + M L_R ds\le M + t\frac{M+\|\Nn(0)\|_\infty}{T_M}\le R.
        \end{aligned}\]
        Consequently, we have $\Phi: Y\to Y$. Furthermore, it follows that
        \[\|(\Phi u)(t)-(\Phi v)(t)\|_\infty\le L_R\int^t_0\|u(s)-v(s)\|_\infty ds\le T_M L_R d(u,v)\le \frac{1}{2} d(u,v).\]
        Thus, $\Phi$ is a contraction in $Y$ with Lipschitz constant $\frac12$, and so $\Phi$ has a unique fixed point $u\in Y$ by the Banach fixed-point Theorem. Let $T=T_M$ we finish the proof.
    \end{proof}
\end{theorem}

This is a standard application of semigroup methods in $C_0(\Omega)$, which treats the problem as an abstract semilinear evolution equation in a Banach space. For later use, we recall the following general result due to Cazenave \cite{existence_mild},  which describes the maximal interval of existence and the blow-up alternative in this abstract setting.

\begin{proposition}\label{global_mild}
   Let $Y$ be a Banach space. There exists a function $T: Y \to (0, \infty]$ with the following
properties: for all $u_0 \in Y$, there exists $u \in C([0, T(u_0)), Y)$ such that for all
$0 < T < T(u_0)$, u is the unique solution of \eqref{mild_sol} in $C([0, T], Y)$. In addition,
\begin{equation}
    2L(\|\Nn(0)\|_Y + 2\|u(t)\|_Y) \ge\frac{1}{T(u_0)-t}-2,~~\forall t \in [0, T(u_0)).\label{10}
\end{equation}
In particular, we have (i) $T(u_0) = \infty$ or (ii) $T(u_0)<\infty$ and $\lim_{t\to T(u_0)}\|u(t)\|_Y=\infty$.
\end{proposition}

 If we further assume that $u\in C([0,T],\Hh(\Omega))\cap C^1([0,T],C_0(\Omega))$, then it can be checked that \eqref{mild_sol} solves \eqref{ori_eq}. Therefore, we call the solution written in \eqref{mild_sol} a mild solution.

\begin{theorem}\label{global existence}
    Let $u_0\in C_0(\Omega)$ and $u\in C([0,T],C_0(\Omega))$. Assume that $\Nn:C_0(\Omega)\to C_0(\Omega)$ is locally Lipschitz continuous. Then $u$ solves \eqref{ori_eq} iff $u$ satisfies \eqref{mild_sol}. 
    Furthermore, if there exist $m,K>0$ such that $u\Nn(u)\le K u^2$ for all $|u|\ge m$, then the weak solution $u$ of \eqref{mild_sol} exists for $\forall u_0\in C_0(\Omega)$ with $t\in(0,\infty)$.
    \begin{proof}
($\Rightarrow$) Assume \(u\) satisfies \eqref{ori_eq}, then for fixed \(0<t<T\), set
\[
\psi(s)=Q(t-s)\,u(s),\quad 0\le s\le t.
\]
Since \(Q(t) = e^{-\epsilon^2\flp t}\) is analytic and \(u\in C^{1}((0,T],C_0(\Omega))\cap C([0,T],\Hh(\Omega))\), then 
\(\psi\in C^{1}([0,t],\Hh(\Omega))\) and
\[
\psi'(s)
=\epsilon^2\flp Q(t-s)u(s)
+Q(t-s)\bigl[\partial_{s}u(s)\bigr]
=Q(t-s)\,\Nn\bigl(u(s)\bigr).
\]
Integrating from \(0\) to \(t\) and using \(Q(0)=I\) yields
\[
u(t)-Q(t)u_{0}
=\int_{0}^{t}Q(t-s)\,\Nn\bigl(u(s)\bigr)\,ds.
\]
($\Leftarrow$) 
Assume \(u\) satisfies \eqref{mild_sol}. Since $\Delta$ is a generator of a analytic semigroup on $C_0(\Omega)$, the subordination theorem implies that $Q(t)$ is a contraction analytic semigroup on $C_0(\Omega)$ \cite{sectorial_and_analytic,pazy_ana_semigroups} and we have the following estimation:

\[
\|\epsilon^2\flp Q(s)\|_{C_0(\Omega)\to C_0(\Omega)}
\le \frac{C}{s},
\qquad s>0.
\]
Therefore, for $t>0$ and $h>0$,
\[
\begin{aligned}
\|Q(t+h)-Q(t)\|_{C_0(\Omega)\to C_0(\Omega)}
&\le\epsilon^2\int_t^{t+h}
\|\flp Q(s)\|_{C_0(\Omega)\to C_0(\Omega)}\,ds \le
\epsilon^2\int_t^{t+h}\frac{C}{s}\,ds \\
&= \epsilon^2\log (1+h/t)\le  \epsilon^2 C_\theta t^{-\theta}h^\theta,
\end{aligned}
\]
where we used $\log(1+h/t)\le C_\theta (h/t)^\theta$ and $C_\theta$ is a constant with $0<\theta<1$.

 Let $f(t)=\Nn(u(t))$, then from the local Lipschitzness of $\Nn$ we conclude that $f\in C([0,T],C_0(\Omega))$.
Fix $\theta$ and $\delta \in(0, T)$ and denote that $M:=\sup_{t\in[0,T]}\|f(t)\|_\infty$. For $t \in[\delta, T]$ and $h>0$ with $t+h \leq T$, we have:
\[\begin{aligned}
    \|u(t+h)-u(t)\|_\infty& \le\|(Q(t+h)-Q(t)) u_0\|_\infty + \|\int_0^t(Q(t+h-s)-Q(t-s)) \Nn(u(s)) d s\|_\infty\\
&+\|\int_t^{t+h} Q(t+h-s) \Nn(u(s)) d s\|_\infty\\
&\le \epsilon^2 C_\theta h^\theta t^{-\theta}\left\|u_0\right\|_\infty + \epsilon^2\int_0^t C_\theta h^\theta(t-s)^{-\theta}\|f(s)\|_\infty d s  +  \int_t^{t+h}\|f(s)\|_\infty d s\\
&\le \epsilon^2 C_\theta \delta^{-\theta} h^\theta\left\|u_0\right\|_\infty  + \epsilon^2\frac{C_\theta M}{1-\theta} h^\theta t^{1-\theta}  + Mh.
\end{aligned}\]
Thus, for a sufficient small $h\ll1$, there exists a constant $C=C\left(\epsilon^2,\theta, \delta, T, M, u_0\right)$ with
\[
\|u(t+h)-u(t)\|_\infty \leq C h^\theta, \quad t \in[\delta, T].
\]
This gives $u \in C^{0, \theta}([\delta, T] , C_0(\Omega))$. By local Lipschitzness of $\Nn$, there exists $R=\sup_{t\in[\delta,T]}\|u(t)\|_\infty$
such that for all $s,t\in [\delta,T]$,
\begin{equation}
    \|f(t)-f(s)\|_\infty =\|\Nn(u(t))-\Nn(u(s))\|_\infty \leq L_R\|u(t)-u(s)\|_\infty \leq L_R C|t-s|^\theta,\label{bound_Nn}
\end{equation}
where $L_R$ the Lipschitz constant. Hence $\Nn(u(\cdot)) \in C^{0, \theta}((0, T] , C_0(\Omega))$.

Denote that $W(t):=\int_0^t Q(t-s) \Nn(u(s)) d s$, so we have
\begin{equation}
    -\epsilon^2\flp W(t)=-\epsilon^2\int_0^t \flp Q(t-s)(\Nn(u(s))-\Nn(u(t))) d s - \epsilon^2\Nn(u(t))\int_0^t \flp Q(t-s) d s \label{dominated}
\end{equation}
The second term lies in $C_0(\Omega)$ since $\epsilon^2\int_0^t \flp Q(t-s) d s=I-Q(t)$. For the first term, we have 
$$
\|-\epsilon^2\flp Q(t-s)(\Nn(u(s))-\Nn(u(t)))\|_\infty \leq C_1(t-s)^{-1} \cdot C_2|t-s|^\theta=C_1 C_2(t-s)^{\theta-1},\quad t \in[\delta, T],
$$
which is integrable near $s=t$ since $\theta \in(0,1)$. Hence $W \in \Hh(\Omega)$, and therefore
\[
u(t)=Q(t) u_0+W(t) \in \Hh(\Omega), \quad t \in(0, T].
\]

Moreover, using dominated convergence in \eqref{dominated} (the same bound $(t-s)^{\theta-1}$ is integrable and $\Nn(u(s)) \rightarrow \Nn(u(t))$ as $s \rightarrow t)$, we get
$$
t \mapsto u(t) \in C((0, T] , \Hh(\Omega)).
$$

To show that $W\in C^1([\delta,T],C_0(\Omega))$ and 
\[W'(t) = -\epsilon^2\int^t_0 \flp Q(t-s)\Nn(u(s))ds + \Nn(u(t)),\]
we start with the existence of 
\begin{equation}
    \frac{W(t+h)-W(t)}{h}=\int_0^t \frac{Q(t+h-s)-Q(t-s)}{h} \Nn(u(s)) d s+\frac{1}{h} \int_t^{t+h} Q(t+h-s) \Nn(u(s)) d s.\label{8}
\end{equation}
Since $\frac{d}{dt}Q(t) = -\epsilon^2\flp Q(t)$, then by changing variable $\tau=t-s$, we write \eqref{8} as:
\[\begin{aligned}
\frac{W(t+h)-W(t)}{h}&= \int_0^t\left(\int_0^1 -\epsilon^2\flp Q(\tau+\eta h) d \eta\right)(f(t)-f(t-\tau)) d \tau \\
   & - \left(\int_0^t \int_0^1 -\epsilon^2\flp Q(\tau+\eta h) d \eta d \tau\right) f(t)
    +\frac{1}{h} \int_t^{t+h} Q(t+h-s) f(s) d s.
\end{aligned}\]
For the first integral, we have $\left\|\int_0^1 -\epsilon^2\flp Q(\tau+\eta h) d \eta\right\|_{C_0(\Omega)\to C_0(\Omega)} \leq C \tau^{-1}$, then by the dominated convergence theorem and \eqref{bound_Nn}, it tends to $-\epsilon^2\int_0^t \flp Q(t-s)(f(s)-f(t)) d s$ as $h \downarrow 0$. For the second and the third integral:
\[
\int_0^t \int_0^1 -\epsilon^2\flp Q(\tau+\eta h) d \eta d \tau \underset{h \downarrow 0}{\longrightarrow} -\int_0^t \epsilon^2\flp Q(\tau) d \tau= Q(t) - I,
\]
\[\frac{1}{h} \int_t^{t+h} Q(t+h-s) f(s) d s  = \frac{1}{h}\int^h_0Q(r)f(t+h-r)dr\underset{h \downarrow 0}{\longrightarrow} f(t),\] 
by strong continuity of $Q$ and $f$ at $t$. Thus
\[
W'(t)= -\epsilon^2\int_0^t \flp Q(t-s)f(s)\,ds + f(t).
\]
Continuity \(t\mapsto W'(t)\) on \([\delta,T]\) follows from the same dominated-convergence argument, so, \(W\in C^1([\delta,T],C_0(\Omega))\) gives $u \in C^1([\delta, T] , C_0(\Omega))$.

\textbf{Global existence.}
Taking the $L^2$-inner product of \eqref{ori_eq} gives 
\[
\frac12\frac{d}{dt}\|u(t)\|_2^2 + \epsilon^2\langle \flp u(t),u(t)\rangle = \int_\Omega \Nn(u(t,x))\,u(t,x)dx
\]
for $t\in(0,T]$. Split the domain into $\{|u|< m\}$ and $\{|u|\ge m\}$ and from  
\[
M:=\sup_{|u|< m}\,\bigl|u \Nn(u)\bigr|<\infty.
\]
By the hypothesis $u \Nn(u)\le K u^2$ for $|u|\ge m$,
\[
\int_\Omega \Nn(u)u dx \le K\|u\|_2^2 + M\,|\Omega|.
\]
Then from Lemma \ref{lemma_spectrum_split}, we find
\[
\frac{d}{dt}\|u(t)\|_2^2 + 2\epsilon^2\lambda_1^{\alpha}\|u(t)\|_2^2\le 2K\|u(t)\|_2^2 + 2M|\Omega|,
\]
and the Grönwall inequality yields, for all $t\in[0,T)$,
\[
\|u(t)\|_2^2
\le
e^{2(K-\epsilon^2\lambda_1^\alpha)t}\|u_0\|_2^2
+
2M|\Omega|
\int_0^t e^{2(K-\epsilon^2\lambda_1^\alpha)(t-s)}\,ds .
\]
Therefore no $L^2$-norm blow-up can occur in finite time and thus the solution $u$ exists globally in a weak sense based on Proposition \ref{global_mild}.  
\end{proof}
\end{theorem}

\begin{remark}
It is observed that $\|u\|_2$ is bounded by $M$ and $\Omega$ if $\epsilon^2\lambda_1^\alpha\ge K$. In particular, if $\epsilon^2\lambda_1^\alpha> K$ then for any $u\in C_0(\Omega)$ solves \eqref{ori_eq}, we have $\sup_{t\ge0}\|u(t)\|_\infty<\infty$ following a similar analysis in \cite{existence_mild}.
\end{remark}
\begin{remark}
    For the case $\alpha\in(\frac{1}{2},1)$, we have $\dom(\flp) = H^{2\alpha}(\Omega)\cap H^1_0(\Omega)$. Then $\|Q(t-s)\Nn(u(s))\|_2$ is bounded by
    \(\|Q(t-s)\Nn(u(s))\|_{H^1}\le C~(1+(t-s)^{-\frac{1}{2\alpha}})\in L^1(0,t)\). The results of ordinary semilinear PDEs could be directly applied to this problem, see \cite{existence_mild}, which gives $u\in C((0,T],H^{2\alpha}(\Omega)\cap H^1_0(\Omega))\cap C^1((0,T],C_0(\Omega))$.
\end{remark}

\section{Bistable reaction--nonlocal diffusion equation}

In this section, we apply the well-posedness theory developed above to a scalar reaction--nonlocal diffusion equation with a bistable Nagumo/Allen--Cahn nonlinearity. Such nonlinearities are standard in models of phase separation and bistable interface dynamics, related space-fractional Allen--Cahn equations have been used to study phase separation and traveling waves \cite{sohaib2024space,chan2017traveling}. Here we use this scalar model as a prototype on bounded domains with spectral fractional diffusion and inhomogeneous boundary conditions. Specifically, we consider the following RNDE:

\begin{equation}
\left\{\begin{aligned}
      &  \partial_t u = -\epsilon^2\flpg u -  \frac{d}{du}F(u), \ (x,t)\in\Omega\times(0,T],\\
      &   u|_{\partial\Omega} = u_-,
\end{aligned}\right.\label{bistable}
\end{equation}
where $F(u)$ is the potential and $\Omega\subset\RR^2$. For the bistable structure, a typical choice is to take the function $F$ such that $\frac{d}{du}F(u)$ has three real zeros, for example, $ \frac{d}{du}F(u) = \frac{1}{\delta} (u-u_-)(u-u_m)(u-u_+)$ with $u_-<u_m<u_+$, where two stable zeros are $u_-$ and $u_+$ and one unstable zero is $u_m$, $\delta$ here is a constant.  Two classical examples are 
\begin{itemize}
  \item the \emph{Nagumo nonlinearity}
  \[
  \frac{d}{du}F(u) = \frac{1}{\delta}u(u-a)(u-1),\qquad\delta>0, a\in(0,1),
  \]
  \item the \emph{Allen--Cahn nonlinearity}
  \[
  \frac{d}{du}F(u) = u(1-u^2),\qquad\delta = -1.
  \]
\end{itemize}

Since $ \frac{d}{du}F(u)$ is a cubic polynomial, it satisfies the growth and local Lipschitz conditions imposed in Section 3, so that all the local well--posedness and regularity results for mild solutions apply to \eqref{bistable}. 

\subsection{Energy functional}
Using relation \eqref{theorem44_18}, we introduce the shifted variable $\tilde{u}=u-u_-$. For simplicity, we still denote $\tilde{u}$ by $u$. With this change of variables, the original RNDE \eqref{bistable} can be rewritten in a homogeneous form.
\begin{equation}
\left\{\begin{aligned}
      &  \partial_t u = -\epsilon^2\flp u -  \Nn(u), \ (x,t)\in\Omega\times(0,T],\\
      &   u|_{\partial\Omega} = 0,
\end{aligned}\right.\label{nagumo}
\end{equation}
with $\Nn(u) = u\,(u+u_- - u_m)(u+u_- - u_+)$. We define the energy functional as
\[
E[u]=\frac{\epsilon^2}{2}\int_{\Omega}\bigl|(-\Delta)^\frac{\alpha}{2}u\bigr|^{2}\,dx
     +\int_{\Omega}\int^u_0\Nn(s)\,ds\,dx.
\]
By applying the calculus of variations we have
\[
     \Bigl\langle\frac{\delta E}{\delta u},\varphi\Bigr\rangle
      =\int_\Omega \bigl(\epsilon^2(-\Delta)^\frac{\alpha}{2} u\,(-\Delta)^\frac{\alpha}{2}\varphi+\Nn(u)\varphi\bigr)\,dx
      =\int_\Omega \bigl(\epsilon^2\flp u + \Nn(u)\bigr)\varphi\,dx,
\]
and thus,
\[
\partial_t u=-\frac{\delta E}{\delta u}.
\]

For a weak solution satisfying the spectral definition of the Laplacian, one shows by testing with $\partial_t u$ and integrating by parts that
\begin{equation}
    \frac{d}{dt}E[u(t)] = \int_{\Omega} (-\Delta)^\frac{\alpha}{2}u\, (-\Delta)^\frac{\alpha}{2}\partial_t u \,dx + \int_{\Omega} \Nn(u)\,\partial_t u\,dx
= -\int_{\Omega} \bigl(\partial_t u\bigr)^2\,dx \le 0.\label{15}
\end{equation}
Hence the energy decays in time:
\begin{equation}
    E[u(t_2)] + \int_{t_1}^{t_2} \|\partial_s u(s)\|_2^2\,ds
= E[u(t_1)], \quad \forall\,0\le t_1< t_2.\label{16}
\end{equation}
Since $\displaystyle\int^u_0\Nn(s)\,ds$ has a lower bound, let $C_b := \inf_{r\in\mathbb{R}}\int_0^r \Nn(s)\,ds > -\infty,$ then
\begin{equation}
    E[u(t)] \ge \frac{\epsilon^2}{2} \int_{\Omega} \bigl|(-\Delta)^\frac{\alpha}{2} u\bigr|^2\,dx +C_b|\Omega|
    \ge \frac{\epsilon^2\lambda_1^\alpha}{2}\int_{\Omega} | u|^2\,dx +C_b|\Omega|.\label{17}
\end{equation}
This implies that the energy-space norm of the weak solution remains uniformly bounded for all time. More precisely, by the equivalence between the $\mathbb{H}^\alpha$-norm and the $H^\alpha$-norm, we obtain
\[
\|u(t)\|_{H^{\alpha}(\Omega)}^2 \le c_1\frac{2+2\lambda^{-\alpha}_1}{\epsilon^2}(E[u(t)]-C_b|\Omega|) \le c_1\frac{2+2\lambda^{-\alpha}_1}{\epsilon^2}(E[u(0)]-C_b|\Omega|), \quad \forall\,t\ge0,
\]
for some constant $c_1>0$. Thus the energy functional $E[u(t)]$ provides both the monotonic decay of the total energy and uniform $H^\alpha$-bounds on weak solutions, i.e.,\[\sup_{t\ge0}\|u(t)\|_{\mathbb{H}^\alpha(\Omega)}\le\sup_{t\ge0}\|u(t)\|_{H^\alpha(\Omega)}<\infty.\]
If we further assume that $u\in C\bigl((0,\infty),\mathbb{H}^{\alpha}(\Omega)\bigr)$, then the following result shows that any $\omega$-limit point of $u(t)$ in $L^2(\Omega)$ is a stationary solution.

\begin{proposition}\label{omega_set}
    For any initial condition $u_0 \in C_0(\Omega)$, define the $\omega$-limit set as
    \[
      \omega(u_0) = \{y\in L^2(\Omega): \exists\, t_n\to\infty,\ u(t_n)\to y \text{ in } L^2(\Omega)\}.
    \]
    Then 
    \begin{itemize}
      \item[(i)] $\omega(u_0) \ne \varnothing$;
      \item[(ii)] any $y\in \omega(u_0)$ is a stationary solution, i.e.,$-\epsilon^2\flp y -  \Nn(y) = 0 \quad \text{in }\mathbb{H}^{-\alpha}(\Omega).$
    \end{itemize}
    \begin{proof}
    Since $E[u(t)]$ has a lower bound, \eqref{15} implies that $\lim_{t\to\infty}E[u(t)] =: E^\infty$ exists. By taking $t_1=0$ and letting $t_2\to\infty$ in \eqref{16}, we conclude that $\int^\infty_0\|\partial_t u(t)\|_2^2\,dt <\infty$.
    Therefore, there exists a sequence $\tau_n\to\infty $ such that $\|\partial_t u(\tau_n)\|_2\to 0$.

    Since $\sup_{t\ge0}\|u(t)\|_{\mathbb{H}^\alpha(\Omega)}<\infty$ and the embedding $\mathbb{H}^{\alpha}(\Omega)\hookrightarrow L^2(\Omega)$ is compact, we obtain that  $\{u(t):t\ge0\}\text{ is relatively compact in } L^2(\Omega)$.
    Then, for any sequence $\{t_n\}$, the sequence $\{u(t_n)\}$ admits a strongly convergent subsequence in $L^2(\Omega)$. In particular, there exist $y\in L^2(\Omega)$ and a subsequence $\{\tau_{n_k}\}\subset\{\tau_n\}$ such that $\|u(\tau_{n_k})-y\|_2\to 0$ as $\tau_{n_k}\to\infty$. Thus $y\in \omega(u_0)$, which proves (i).

    For any $y\in\omega(u_0)$, there exists a sequence $\{t_n\}$ such that $t_n\to\infty$ and $u(t_n)\to y$ in $L^2(\Omega)$. Letting $n\to \infty$, and from \eqref{16} with $t_1=t_n$ and $t_2=t_n+1$ we obtain
    \[
      \int^{t_n+1}_{t_n}\|\partial_t u(t)\|_2^2\, dt \to 0.
    \]
    Hence, by the mean-value theorem, for each $n$ there exists $s_n\in[t_n,t_n+1]$ such that 
    \[
      \|\partial_t u(s_n)\|_2^2\le\int^{t_{n}+1}_{t_n}\|\partial_t u(t)\|_2^2\, dt \to 0,
    \]
    and 
    \[
      \|u(s_n)-u(t_n)\|_2
      \le\int^{s_n}_{t_n}\|\partial_t u(t)\|_2\,dt
      \le\Bigl(\int^{t_{n}+1}_{t_n}\|\partial_t u(t)\|_2^2\, dt\Bigr)^{1/2}\to 0.
    \]
    This shows that $u(s_n)\to y$ in $L^2(\Omega)$. Moreover, the sequence $\{u(s_n)\}$ is bounded in $\mathbb{H}^\alpha(\Omega)$, so, up to a further subsequence (still denoted by $\{s_n\}$), we have $u(s_n)\rightharpoonup y \quad\text{weakly in }\mathbb{H}^\alpha(\Omega)$, and in particular $y\in\mathbb{H}^\alpha(\Omega)$.

    Now, for any $\phi\in\mathbb{H}^\alpha(\Omega)$, we have the weak form of \eqref{nagumo} at $t=s_n$:
    \[
      (\partial_t u(s_n),\phi) 
      + \epsilon^2\bigl((-\Delta)^{\alpha/2} u(s_n),(-\Delta)^{\alpha/2}\phi\bigr) 
      + (\Nn(u(s_n)),\phi)=0.
    \]
    Using $\|\partial_t u(s_n)\|_2\to0$ and the weak convergence $u(s_n)\rightharpoonup y$ in $\mathbb{H}^\alpha(\Omega)$, we can pass to the limit in the bilinear form involving $(-\Delta)^{\alpha/2}$. Moreover, since $\Nn$ is locally Lipschitz and $u(s_n)\to y$ in $L^2(\Omega)$, then we have $\Nn(u(s_n))\to \Nn(y)$ in $L^2(\Omega)$. 
    \[
      \epsilon^2\bigl((-\Delta)^{\alpha/2} y,(-\Delta)^{\alpha/2}\phi\bigr) + (\Nn(y),\phi)=0 \quad \text{for any }\phi\in \mathbb{H}^{\alpha}(\Omega),
    \]
   by letting $n\to \infty$.
    \end{proof}
\end{proposition}

Specifically, $C_b=0$ when $u_m\ge \frac{u_-+u_+}{2}$, then from \eqref{17} the coercivity estimate is given by
\[
\|u(t)\|_{\mathbb{H}^{\alpha}(\Omega)}^2 \le c_1\,\frac{2+2\lambda^{-\alpha}_1}{\epsilon^2}\,E[u(0)] \quad \text{for } t\in(0,\infty).
\]
If we increase the regularity of $u(t)$, for example $u(t)\in D((-\Delta)^\alpha)=\Hh(\Omega)$, then the stationary equation in Proposition~\ref{omega_set} actually holds in $L^2(\Omega)$ rather than only in $\mathbb{H}^{-\alpha}(\Omega)$.

\subsection{Boundedness of solution}
Let $P=\{(x, t) \mid x \in \Omega, 0<t \leq T\}, \Gamma$ is the parabolic boundary of $P$, which is given by $\Gamma=\partial \Omega \times(0, T) \bigcup \Omega \times\{t=0\}$. We denote by $C^{1,1}(P)$ the set of functions which are continuously differentiable with respect to $x$ and continuously differentiable with respect to $t$.

\begin{lemma}\label{maximum}
    For a function $u(x)\in \Hh(\Omega)$ and $\flp u\in C(\Omega)$ with $\alpha\in(0,1)$ . If $x^*$ is a global maximum of $u(x)$, then $\flp u(x^*)\ge0$. If $x^*$ is a global minimum of $u(x)$, then $\flp u(x^*)\le0$.
\end{lemma}
\begin{proof}
    Since that $e^{t\Delta}$ is a positive-preserved and $L^\infty$-contractive semigroup. Then for the maximum $x^*$, it gives
    \[(e^{t\Delta}u)(x)\le\|u\|_\infty=u(x^*),\ \forall x\in\Omega.\]
    Therefore, $\flp u(x^*)\ge0$. And let $v(x) = -u(x)$ we can similarly prove $\flp u(x^*)\le0$.
\end{proof}

\begin{lemma}[Weak Maximum Principle]\label{weak_maximum}
Assume that $k: P \to \mathbb{R}$ is a bounded continuous function and there exists a constant $k_0 > 0$ such that
    \[
    k(x, t) \ge k_0 > 0 \quad \text{for all}\quad (x, t) \in P.
    \]
Suppose $u \in C^{1,1}(P) \cap C(\overline{P})\cap C([0,T],\Hh(\Omega))$ satisfies the differential inequality
\begin{equation}\label{weak_maximum_1}
    k(x,t)\partial_t u + \mathcal{L} u + c(u)u \geq 0 \quad \text{in } P,
\end{equation}
and the boundary condition
\[
\min_{\Gamma} u(x,t) \ge 0,
\]
where $\mathcal{L}:=\epsilon^2\flp +\nabla$ and $c(u)\ge-c_0$ with a constant $c_0\in \RR^+$. Then
\[
\min_{\overline{P}} u(x, t) \geq 0.
\]     
\begin{proof}
Choose
\[
  \lambda:=\frac{c_0}{k_0},
  \qquad w(x,t):=e^{-\lambda t}u(x,t).
\]
Substituting into \eqref{weak_maximum_1} and dividing by $e^{\lambda t}>0$ yields
\begin{equation}
  k\,\partial_t w + \mathcal{L}w + a(x,t,w)w \;\ge\; 0 \label{new_weak_maximum}
  \quad\text{in }P,
\end{equation}
where $a(x,t,w)=\lambda k + c(e^{\lambda t}w)$, by $k\ge k_0$ and $c(w)\ge -c_0$ we have $a(x,t,z)\ge0$, for every $(x,t)\in P$ and every real $w$.
Moreover, on $\Gamma$ we have $w=e^{-\lambda t}u\ge 0$.

Now we claim that \eqref{new_weak_maximum} plus $w|_{\partial\overline{P}}\ge 0$ implies $w\ge 0$ on $\overline P$.
Assume by contradiction that $\min_{\overline P} w<0\le\min_{\partial\overline P} w$.
Define that 
\[
  w_\varepsilon^2(x,t):=w(x,t)+\varepsilon^2 t,
\]
where $0<\varepsilon^2\ll 1$.

It is clear to see that that $w_\varepsilon^2\ge w\ge 0$ on $\partial\overline{P}$. By assumption, there exists a point $(x^*,t^*)\in P$ and a $\delta>0$ such that $\min_{\overline{P}}w\le-\delta<0$. Then for $\forall \delta>0$ we choose $\varepsilon^2<\frac{\delta}{T}$ such that
\[w_\varepsilon^2(x^*,t^*) = w(x^*,t^*) + \varepsilon^2 t^*<0\]

Hence $\min_{\overline{P}}w_\varepsilon^2<0$ is attained at an interior point $(x_\varepsilon^2,t_\varepsilon^2)\in P$. Therefore, by minimality $\partial_t(w_\varepsilon^2)(x_\varepsilon^2,t_\varepsilon^2)\le0$ and the Lemma \ref{maximum},
\[\nabla w_\varepsilon^2(x_\varepsilon^2,t_\varepsilon^2)=0,\ \flp w_\varepsilon^2(x_\varepsilon^2,t_\varepsilon^2)\le0,\  \partial_t(w_\varepsilon^2)(x_\varepsilon^2,t_\varepsilon^2)\le0.\]
This is gives
\begin{equation}
    k(x,t)\partial_t(w_\varepsilon^2)+\mathcal{L}w_\varepsilon^2+a(x,t,w)w^2_\varepsilon\le0\ \ \text{at }(x_\varepsilon^2,t_\varepsilon^2).\label{contradicts}
\end{equation}
On the other hand,
\[
  k\partial_t(w_\varepsilon^2) + \mathcal{L}w_\varepsilon^2 + a\,w_\varepsilon^2
  = \big(k \partial_tw + \mathcal{L}w + a\,w\big) + \varepsilon^2 k + \varepsilon^2 t\,a,
\]
and using \eqref{new_weak_maximum} together with $k\ge k_0>0$, $a\ge 0$, $t\ge 0$, we get
\begin{equation}\label{neg-side}
  k\partial_t(w_\varepsilon^2) + \mathcal{L}w_\varepsilon^2 + a\,w_\varepsilon^2
  \;\ge\; \varepsilon^2 k(x_\varepsilon^2,t_\varepsilon^2)
  \;\ge\; \varepsilon^2 k_0 > 0
  \quad\text{at }(x_\varepsilon^2,t_\varepsilon^2),
\end{equation}
which contradicts \eqref{contradicts}. Therefore $\inf_{\overline P}w\ge 0$, i.e. $w\ge 0$ on $\overline P$.
Since $u=e^{\lambda t}w$ with $e^{\lambda t}>0$,
we conclude $u\ge 0$ in $\overline P$.
\end{proof}
\end{lemma} 
\begin{remark}
This lemma introduces no new hypotheses. It restates the nonlinear growth condition in Theorem \ref{global existence} in terms of $c(u)$ by writing $\Nn(u)=-c(u)u$ and $C=-c_0$, and it upgrades the $L^2$ estimate there to a uniform $L^\infty$ bound.
\end{remark}

Directly applying Lemma \ref{weak_maximum} yields the following theorem on the boundedness of the solution.
\begin{theorem}\label{bound_bistable}
    The solution of the RNDE \eqref{nagumo} satisfies $0 \leq u(x,t) \leq u_+-u_-$ in $\bar{P}$ for $0 \leq u_0(x) \leq u_+-u_-$ with $u_0(x):=u(x,0)$.
    \begin{proof}
It is straightforward to deduce from Lemma \ref{weak_maximum} that $u(x,t)\ge0$ whenever $u_0(x)\ge0$.  We here show that $u(x,t)\le u_+-u_-$. Firstly, we take auxiliary function $v(x, t)=u_+-u_- - u(x, t)$, then the relation \eqref{theorem44_18} gives $v(x, t)$ satisfying the following equation:
\begin{equation}
   \left\{ \begin{aligned}
    \partial_t v + \epsilon^2& \flpg v - v(u_+-u_--v)\,(u_+ - u_m - v)  =0, & & \text { in } \Omega \times(0, T], \\
v(x, t) & =u_+-u_-, & & \text { on } \partial \Omega \times(0, T], \\
v(x, 0) & =u_+-u_- -u_0(x), & & \text { on } \Omega.
\end{aligned}\right.
\end{equation}
Use the relation \eqref{theorem44_18}, it is clear to see that $\flpg v(x^*)\ge0~(\le0)$ when $x^*$ is the global maximum (minimum) of $v(x)$ by using Lemma \ref{maximum}. From $u_0(x) \leq u_+ - u_-$ we obtain $v(x, t) \geq 0$ on $\Gamma$. Assume that $\min_{\overline{P}}v<0$, clearly,  this can only be attained in $P$. Define  
\[t^* = \min_{x\in{\Omega}}\{t\in(0,T):v(x,t)<0\}\]
which is also the first time $u(x,t)>u_+-u_-$. Denote $x^*$ is the maximal point of $u(x,t^*)$. Then, by definition, 
\[\partial_t v(x^*,t^*)\ge 0,~~\flpg v(x^*,t^*)\ge 0,\]
it gives
\[    0 \ge\partial_t v \ne -\epsilon^2 \flpg v + v(u_+-u_--v)\,(u_+ - u_m - v) < 0\text{ at }(x^*,t^*).\]
This contradicts the assumption. Thus, we obtain $u(x, t) \le u_+$ in $\overline{P}$.
\end{proof}
\end{theorem}
Returning to the original bistable RNDE \eqref{bistable}, we conclude that the solution remains bounded by the two stable steady states $u_-$ and $u_+$, provided the initial condition lies within this range. Furthermore, it is worth noting that the proof of Theorem \ref{bound_bistable} inherently establishes a generalized weak maximum principle. Consequently, we can extend Lemma \ref{weak_maximum} to the inhomogeneous operator $\flpg$:
\begin{corollary}\label{in_lap_max_principle}
    Suppose $u \in C^{1,1}(P) \cap C(\overline{P})\cap C([0,T],\Hh(\Omega))$ satisfies the differential inequality
\[    k(x,t)\partial_t u + \flpg u + c(u)u \geq 0 \quad \text{in } P,\]
and 
\(\min_{\Gamma} u(x,t) \ge 0,\)
where $c(u)\ge-c_0$ with a constant $c_0\in \RR^+$. Then
\(\min_{\overline{P}} u(x, t) \geq 0.
\)
\end{corollary}

\subsection{Numerical simulations}\label{subsec:numerics_bistable}

 In this part, we perform numerical experiments for the 2D reaction--nonlocal diffusion equation to illustrate how the fractional order $\alpha$ affects propagation and to verify the boundedness property observed in the analysis. We employ a sine pseudospectral discretization (DST-I) on a uniform $N\times N$ grid with $N=1024$.
This choice is natural for homogeneous Dirichlet boundary conditions and diagonalizes the spectral fractional Laplacian.
More precisely, writing the discrete sine transform coefficients $\widehat{u}_{k m}$, the operator $(-\Delta)^\alpha$ acts as a multiplier in spectral space:
\[
\widehat{\,(-\Delta)^\alpha u\,}_{k m} = \lambda_{k m}^{\alpha}\,\widehat{u}_{k m},
\qquad
\lambda_{k m}=\Bigl(\frac{k \pi}{L}\Bigr)^2+\Bigl(\frac{m \pi}{L}\Bigr)^2,\quad 1\le k, m\le N-1.
\]
Time integration is performed by a fourth-order exponential time--differencing Runge--Kutta method (ETDRK4), which is well suited for stiff diffusion operators.
Implementation details, including the evaluation of the ETD coefficients and the spectral treatment of $(-\Delta)^\alpha$, are provided in Appendix~A.

\paragraph{Model and parameters.}
We consider the shifted fractional Nagumo equation in the homogeneous Dirichlet setting,
\begin{equation}\label{eq:numerics_nagumo}
\partial_t u = -\epsilon^2(-\Delta)^\alpha u - \Nn(u), 
\qquad (x,t)\in \Omega\times(0,T],
\qquad u|_{\partial\Omega}=0,
\end{equation}
with $u_-=0$, $u_m=a$, $u_+=1$ and
$\Nn(u)=\delta^{-1}u(u-a)(u-1)$.
Unless otherwise stated, we use the following parameters:
\[
\Omega=(0,L)^2,\quad L=2,\quad
a=0.35,\quad \epsilon^2=5\times10^{-3},\quad \delta=10^{-2},
\quad T=5.
\]
The initial condition $u_0(x,y)$ is a localized perturbation centered at $(1,1)$. More precisely, it is zero everywhere except in a small disk
$D_N:=\{(x,y):(x-1)^2+(y-1)^2\le 0.02^2\}$, where it takes the value $0.5$, that is,
\[
u_0(x,y)=
\left\{
\begin{aligned}
&0.5, \quad (x,y)\in D_N,\\
&0, \quad (x,y)\in \Omega\backslash D_N.
\end{aligned}
\right.
\]
In practice, we use a smooth approximation of this profile so that $u_0\in C_0(\Omega)$.

Figure~\ref{fig:nagumo_alpha} shows the solution of \eqref{eq:numerics_nagumo} at the final time $T=5$ for $\alpha = 0.65, 0.75, 0.85, 0.95$. Increasing $\alpha$ visibly produces faster spreading: the region in the $(x,y)$-plane where $u(\cdot, T)$ is appreciably positive becomes larger, and the profile becomes smoother. This is consistent with the spectral scaling $\lambda_{km}^{\alpha}$, which enhances the damping of high-frequency modes as $\alpha$ increases. Besides, the computed solutions remain within the interval $[0,1]$ for all tested $\alpha$. In particular, although the solution becomes more spatially spread out as $\alpha$ increases, the amplitude does not overshoot the stable states. This numerically reflects the invariant-region property proved in Theorem~\ref{bound_bistable} for the initial condition satisfying $0\le u_0\le 1$.

In the simulations, after an initial transient the evolution slows down, and the solution at each spatial grid point shows a clear tendency to settle: values that are initially close to $0$ increase over time and gradually approach the stable level near $1$ as the activation region expands.
To further quantify this behavior, we measure how the fractional order $\alpha$ influences the spatial spreading of $u$.

\begin{figure}[htbp]
\centering
\begin{minipage}[t]{0.23\textwidth}
    \centering
    \includegraphics[width=\linewidth]{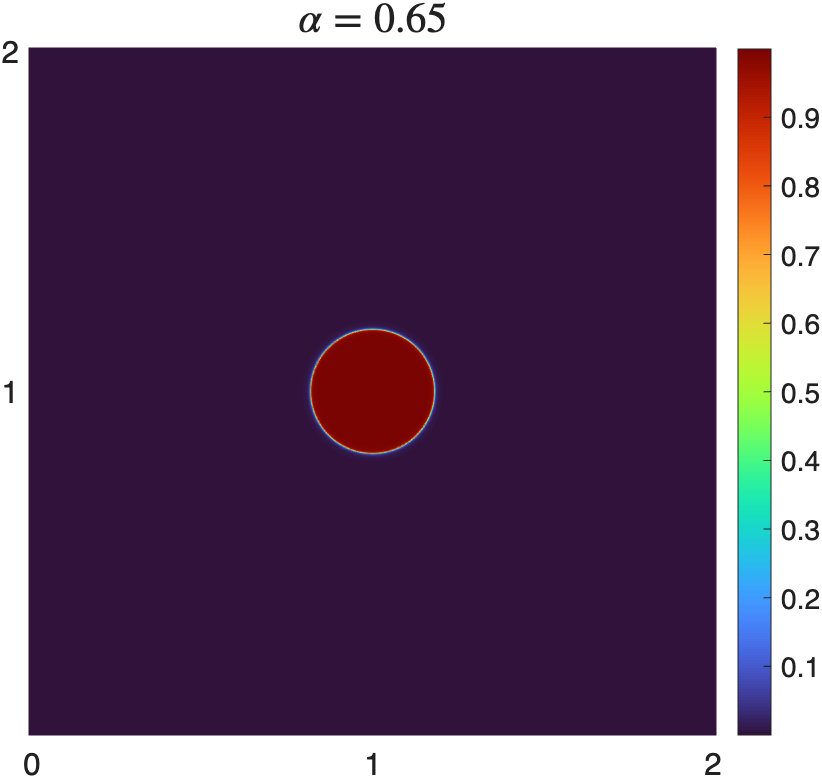}
\end{minipage}
\hfill
\begin{minipage}[t]{0.23\textwidth}
    \centering
    \includegraphics[width=\linewidth]{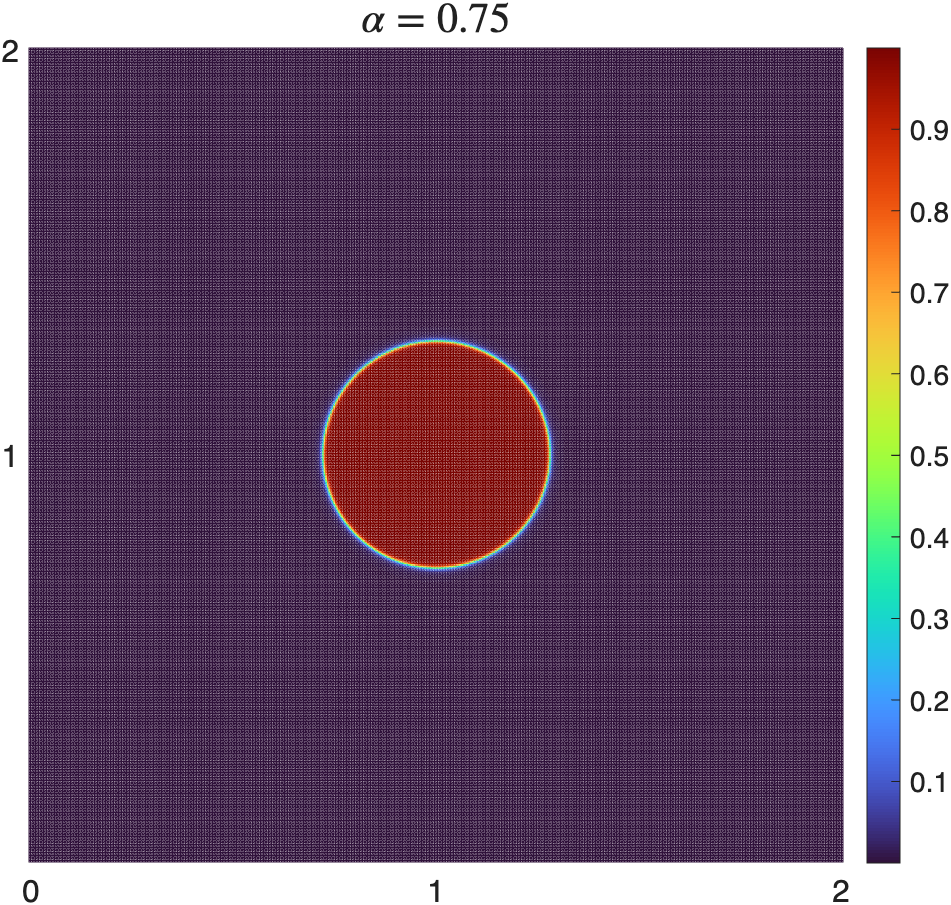}
\end{minipage}
\hfill
\begin{minipage}[t]{0.23\textwidth}
    \centering
    \includegraphics[width=\linewidth]{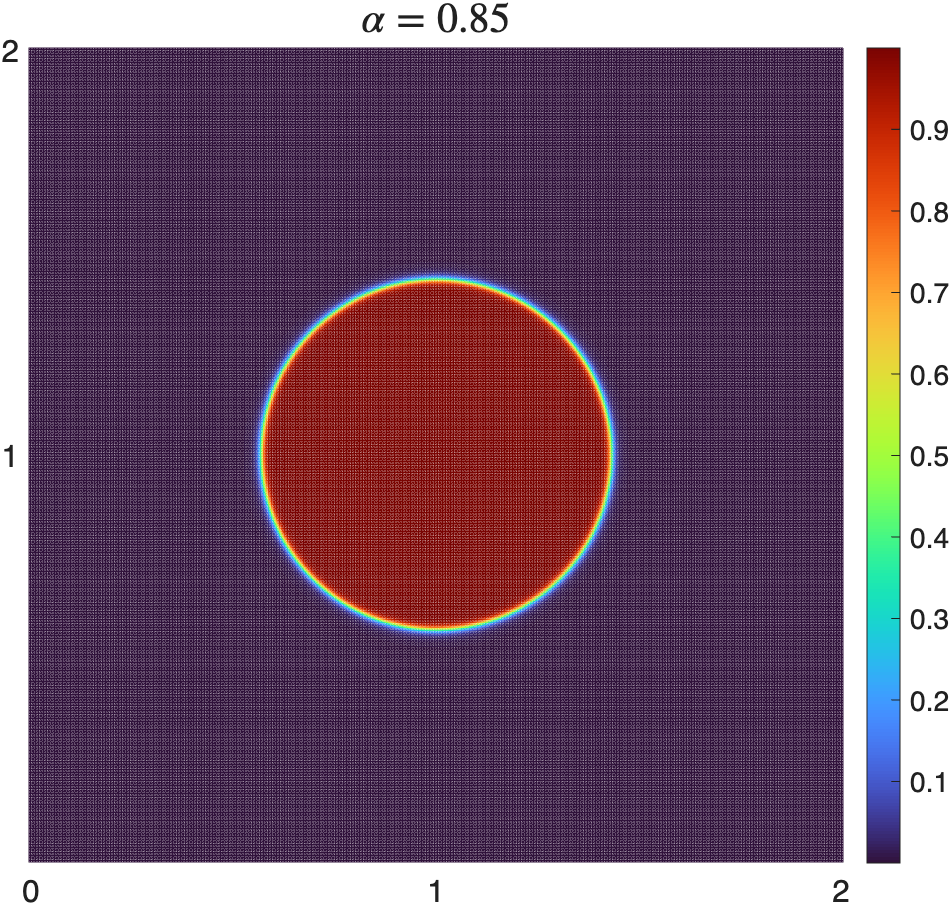}
\end{minipage}
\hfill
\begin{minipage}[t]{0.23\textwidth}
    \centering
    \includegraphics[width=\linewidth]{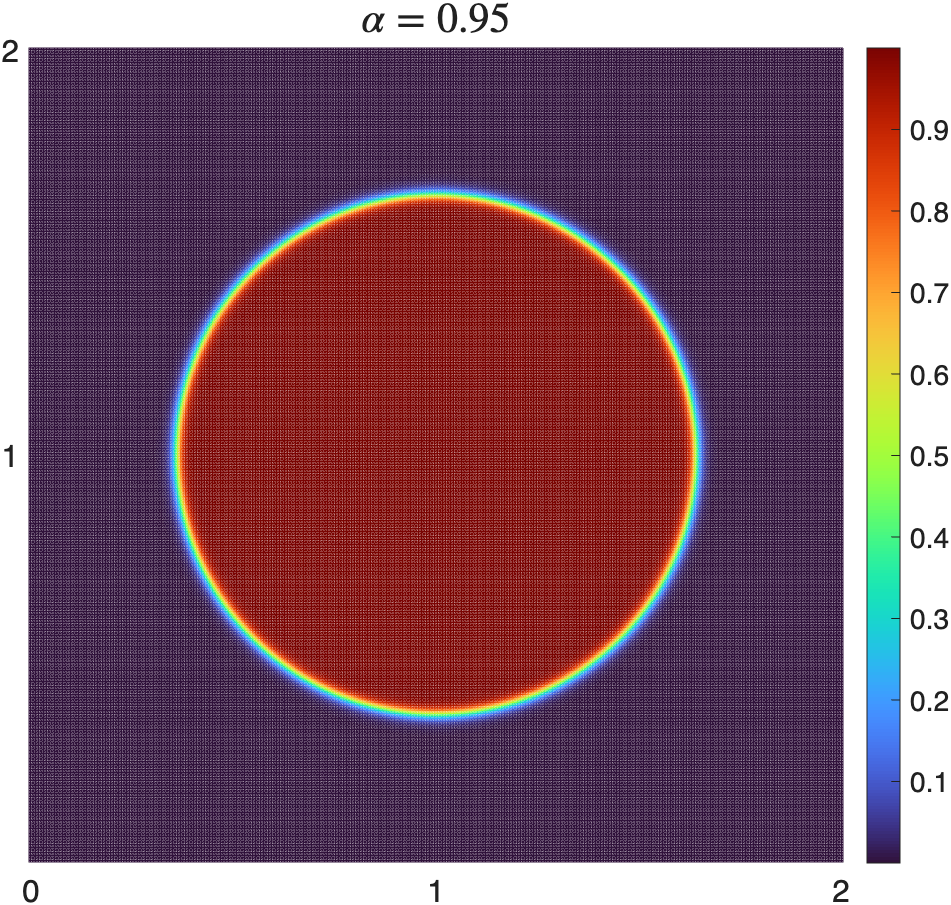}
\end{minipage}
\caption{Numerical solutions of the fractional Nagumo model \eqref{nagumo} on a $1024^2$ grid at $T=5$ for $\alpha=0.65,0.75,0.85,0.95$, where $a=0.35$, $\epsilon^2=5\times 10^{-3}$, and $\delta=10^{-2}$.}
\label{fig:nagumo_alpha}
\end{figure}

The long-time behavior is also compatible with the gradient-flow picture developed above. The energy dissipation identity \eqref{16} suggests relaxation toward stationary states, and Proposition~\ref{omega_set} characterizes $\omega$-limit points as stationary solutions. In the simulations, after an initial transient the evolution slows down and exhibits invasion-type dynamics: points not yet reached by the front stay near the stable state $u=0$, while activated points evolve toward the other stable state $u=1$. To quantify this behavior, we measure how the fractional order $\alpha$ influences the spatial spreading of $u$.

For a fixed threshold $\theta\in(0,1)$, we define the superlevel-set area
\[
A_\theta(t):=\bigl|\{(x_i,y_j)\in\Omega:\ u(x_i,y_j,t)\ge \theta\}\bigr|.
\]

A larger $A_\theta(t)$ means that the solution occupies a larger part of the domain at level $\theta$, i.e., it is more spread out. 
Figure \ref{nagumo_speed} plots $A_{0.5}(t)$ for several values of $\alpha=0.65,0.75,0.85,0.95$ and shows a clear monotone trend: larger $\alpha$ yields a faster increase of $A_{0.5}(t)$, hence a faster spreading rate. Moreover, for all $\alpha$, the curves become smoother at later times and their slopes gradually level off, indicating a nearly constant growth rate.

\begin{figure}
    \centering
    \includegraphics[width=0.3\linewidth]{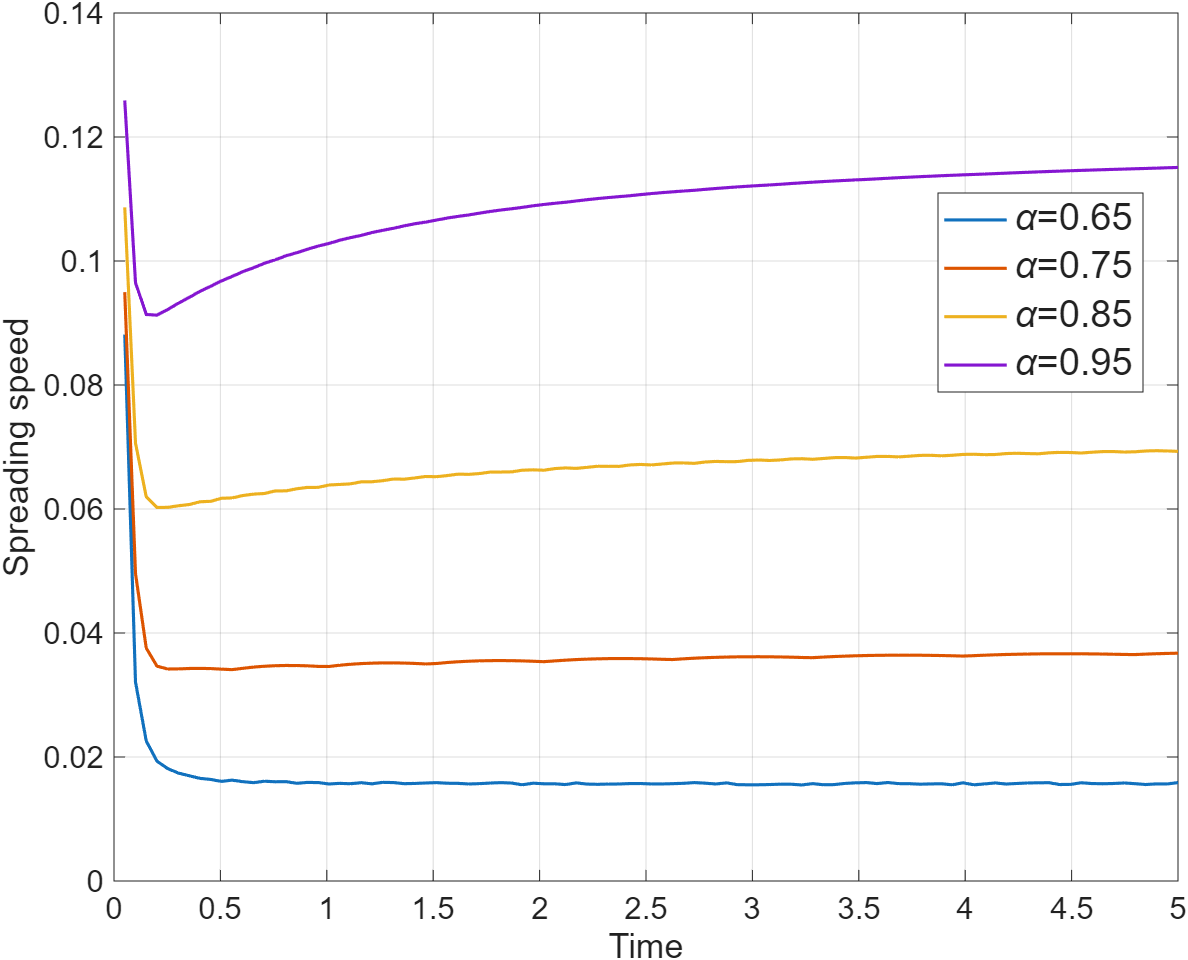}
    \caption{Spreading speed of the fractional Nagumo model \eqref{nagumo} with $\alpha=0.65,0.75,0.85,0.95$, where $a=0.35$, $\epsilon^2=5\times 10^{-3}$, and $\delta=10^{-2}$ are used.  }
    \label{nagumo_speed}
\end{figure}

\section{The fractional Gray-Scott system}

After completing the study of a single nonlocal reaction--diffusion equation, we proceed to coupled nonlocal systems. 
In particular, we consider the Gray–Scott model, which serves as a classical prototype in the theory of pattern formation. Mathematically, this system provides a canonical example of an activator–inhibitor mechanism that exhibits rich spatiotemporal dynamics, including diffusion-driven instabilities and Turing-type patterns. 
In recent years, its space-fractional extension has attracted considerable attention: existing works address well-posedness, discretization, and simulations of fractional Gray--Scott dynamics and demonstrate that the fractional order can substantially modify both instability thresholds and pattern morphologies \cite{yuan2024adaptive,wang2019fractional}. 
Related studies also investigate nonlocal (kernel-based) Gray--Scott models, establishing well-posedness and analyzing diffusive limits toward the classical system \cite{laurenccot2023nonlocal}. In contrast to the integral-type fractional Laplacian used in these works, we here focus on the spectral fractional Laplacian on bounded domains with inhomogeneous Dirichlet boundary conditions. Benefiting from the relation \eqref{theorem44_18} in Section~2.3, we convert it into a homogeneous Dirichlet boundary system. This yields a clean spectral multiplier representation, which enables a unified analysis based on the semigroup approach and leads to an efficient spectral discretization (e.g., DST-I).
Moreover, the analysis developed in this section can be adapted to related RNDEs—such as the Schnakenberg, Gierer–Meinhardt, or Brusselator models—which share the same spectral multiplier representation but differ in their specific nonlinearities.

Consider the following fractional Gray–Scott system:
\begin{equation}\label{ori_gs}
    \left\{\begin{aligned}
    &\partial_t u=-r_u\flpg u-u v^2+F (1-u),  \quad (x, t) \in \Omega \times[0, T],  \\ 
    & \partial_t v=-r_v(-\Delta)^\beta v+u v^2-(F+\kappa) v, \quad (x, t) \in \Omega \times[0, T], \\ 
    & u(x,0)=u_0(x), \quad v(x,0)=v_0(x), \quad x \in \Omega, \\ 
    & u = 1, \quad  v =0, \quad (x, t) \in \partial \Omega \times[0, T],
\end{aligned}\right.
\end{equation}
where $\Omega\subset\RR^2$ is a bounded domain, $r_u$, $r_v$, $F$, $\kappa$ are positive constants. According to the relation \eqref{theorem44_18}, the fractional Gray–Scott system \eqref{ori_gs} can be reduced in a homogeneous RNDE form.
\begin{equation}\label{gs}
    \left\{\begin{aligned}
    &\partial_t u=-r_u(-\Delta)^\alpha u-(u+1) v^2 - F u,  \quad (x, t) \in \Omega \times[0, T],  \\ 
    & \partial_t v=-r_v(-\Delta)^\beta v+(u+1) v^2-(F+\kappa) v, \quad (x, t) \in \Omega \times[0, T], \\ 
    & u(x,0)=u_0(x) - 1, \quad v(x,0)=v_0(x), \quad x \in \Omega, \\ 
    & u = 0, \quad  v =0, \quad (x, t) \in \partial \Omega \times[0, T].
\end{aligned}\right.
\end{equation}

This transformation enforces a consistent and continuous boundary condition up to \(\partial\Omega\). It also allows us to directly use the positivity (and related energy) properties of the homogeneous fractional diffusion operator $\flp$, and it simplifies the numerical discretization by avoiding extra boundary-correction terms. Therefore, in the sequel we restrict our analysis and computations to this homogeneous system.

\subsection{Local boundedness}

\begin{theorem}\label{bound_gs}
    For every positive initial condition $u_0,v_0\in C_0(\Omega)$, there is a unique solution 
    \[u,v\in C([0,T],C_0(\Omega))\]
    to \eqref{gs} and satisfies
    \[-1 \le u(x,t) \le e^{-F t}(1- \|u_0\|_\infty)~\text{ and }~v\ge0\]
if $u_0(x)\le 1$ for $(x,t)\in\Omega\times[0,T)$.
\begin{proof}
The reaction term $\Nn(u,v) = \big(-(u+1)v^2 - Fu, (u+1)v^2 - (F+k)v\big)^T$ has a continuous Jacobian 
\[J(u, v) = \begin{pmatrix} -(v^2 + F) & -2(u+1)v \\ v^2 & 2(u+1)v - (F+\kappa) \end{pmatrix}\] 
consisting of polynomial entries. Thus, on any compact set $D \subset \mathbb{R}^2$ with $|u| \le R$ and $|v| \le R$, the entries are bounded by $R$, $F$, and $\kappa$, and by the Mean Value Theorem 
$\Nn$ is locally Lipschitz. Then, the local existence and uniqueness are directly from Proposition \ref{global_mild} and Theorem \ref{global existence}, i.e., $u,v\in C((0,T),\Hh(\Omega))$. Therefore, the solution $(\widetilde{u},v)$ of \eqref{ori_gs} is also existed.  Since $\partial_t \widetilde{u} + r_u(-\Delta)_g^\alpha \widetilde{u} + \widetilde{u} (v^2+1)=F>0$, we have $\widetilde{u}\ge0$ for $(x,t)\in\Omega\times[0,T]$ from Corollary \ref{in_lap_max_principle}, that is $u\ge-1$ in the RNDE system \eqref{gs}.

To show the boundedness of $u$, we notice that 
\[\begin{aligned}
    u(t)&=e^{-t(r_u\flp
+F)}(u_0 - 1)
  - \int_0^t e^{-(t-s)(r_u\flp
+F)}\bigl[(1+u(s))v(s)^2\bigr]\,ds\\
&\le e^{-t(r_u\flp
+F)}(u_0 - 1)\\
&\le e^{-tF}(u_0 - 1).
\end{aligned}\]
Together with $u\ge -1$, this gives $\|u(t)\|_\infty\le 1$. 

Since $F + \kappa - (u+1)v \ge F + \kappa$, we again apply Lemma \ref{weak_maximum} the $v$-equation and conclude that $v(t) \ge 0$ on $[0,T]$.
\end{proof}
\end{theorem}

Following Theorem \ref{bound_gs}, we obtain a straightforward conclusion: there exists an invariant set for this G-S system \eqref{gs}.

\begin{corollary}[Invariant set]
    Let the initial condition $u_0,v_0\in C_0(\Omega)$. If $\|u_0\|_\infty \le 1$ and $\|v_0\|_\infty \le \frac{F+\kappa}{2-\|u_0\|_\infty}$, then the solution $(u, v)$ exists globally and satisfies
    \[\|u(t)\|_\infty\le 1,~~\|v(t)\|_\infty\le\frac{F+\kappa}{2-\|u_0\|_\infty} ~\text{ for any }t\in(0,\infty).\]

\begin{proof}
Set $M = \frac{F+\kappa}{2-\|u_0\|_\infty}$.
Assume, to the contrary, that $\sup_{\Omega\times[0,T]} v>M$.
Because $v_0\le M$ by hypothesis, we define the first crossing time
\[
t_*:=\inf\Big\{t\in(0,T]:\ \max_{\overline\Omega} v(\cdot,t)>M\Big\}.
\]
By continuity of $t\mapsto v(\cdot,t)$ in $C_0(\Omega)$, we have
\(
\max_{\overline\Omega} v(\cdot,t)\le M
\) for all $t\in[0,t_*]$.
Hence, there exists $x_*\in\overline\Omega$ such that
\[
v(x_*,t_*)=M=\max_{\overline\Omega} v(\cdot,t_*).
\]
Since $v(\cdot,t_*)\in C_0(\Omega)$ and $v=0$ on $\partial\Omega$, then $x_*\in\Omega$. By Lemma \ref{maximum}, we have $\flp v(x_*,t_*)\ge 0$.

For every $t\in[0,t_*]$ we have $v(x_*,t)\le \max_{\overline\Omega}v(\cdot,t)\le M=v(x_*,t_*)$.
Therefore, the function $t\mapsto v(x_*,t)$ attains its maximum over $[0,t_*]$ at $t=t_*$.
Since $v$ is differentiable in time, we conclude that
\begin{equation}\label{eq:dtv_nonneg}
\partial_t v(x_*,t_*)\ge 0.
\end{equation}

Evaluating the $v$-equation at $(x_*,t_*)$ gives
\[
\partial_t v(x_*,t_*)
=-r_v(-\Delta)^\beta v(x_*,t_*)+(u(x_*,t_*)+1)M^2-(F+\kappa)M.
\]
Using \ $-r_v\flp v(x_*,t_*)\le 0$ and from Theorem \ref{bound_gs}, we obtain
\begin{align*}
\partial_t v(x_*,t_*)
&\le (u(x_*,t_*)+1)M^2-M(F+\kappa) \\
&\le\bigl(u(x_*,t_*)+1-1\bigr)M^2
=u(x_*,t_*)\,M^2\le 0,
\end{align*}
which contradicts \eqref{eq:dtv_nonneg}. Consequently,
\begin{equation}
    \|u(t)\|_\infty\le 1,~~\|v(t)\|_\infty\le\frac{F+\kappa}{2-\|u_0\|_\infty}\label{28}
\end{equation}
for all $t\in[0,T]$. By Proposition \ref{global_mild}, the solution \eqref{28} can be extended globally in time.
\end{proof}
\end{corollary}

For the case $r_v\ge r_u$ and $\beta\ge\alpha$, one formally has
\[u+v\le e^{-r_u\flp t}(u_0+v_0)-\int^t_0 e^{-r_u\flp (t-s)}Fu(s)ds,~~\forall t\in(0,T).\]
Combined with \eqref{28}, this gives an upper bound for the $v$-component. However, we normally assume the condition $r_u > r_v$, where $r_u$ and $r_v$ represent the diffusion coefficients of the reactant $u$ and the autocatalyst $v$, respectively. Under this condition, the reactant $u$ diffuses rapidly from the surrounding environment to ``feed'' the local production of $v$. Simultaneously, the slow diffusion of $v$ allows it to remain localized and maintain a high concentration. This mechanism creates a local positive feedback loop that facilitates the growth of spatial patterns. Conversely, if $r_v \ge r_u$, the qualitative dynamics of the system change significantly. In this regime, $v$ diffuses too quickly to accumulate into localized peaks. Furthermore, because the diffusion of $u$ is relatively slow, it cannot replenish the depleted regions effectively. Consequently, complex spatial patterns typically fail to emerge, and the system instead evolves toward a spatially homogeneous state. The same modeling consideration applies to the fractional orders $\alpha$ and $\beta$, and in the pattern-forming regimes below we typically take $\alpha\ge\beta$. This motivates the following interior estimate.

\begin{corollary}\label{wellposedness_gs}
Denote $(\lambda_1^\alpha, \phi_1)$ as the first eigenpair of the fractional Laplacian $(-\Delta)^\alpha$ with Dirichlet-0 boundary condition, i.e.,
\[
\flp\phi_1 = \lambda_1^\alpha \phi_1, \quad \phi_1|_{\partial\Omega} = 0,
\]
where $\phi_1$ is $L^2$-normalized. If $u_0 \in L^2(\Omega)$ with $\|u_0\|_\infty\le 1$, then the following holds:

\medskip
\noindent{\rm (i)} For any $t \in [0,T]$,
\begin{equation}
\int_0^t e^{-r_u\lambda_1^\alpha(t-s)} \left( \int_\Omega v(x,s)^2 \phi_1(x) \, dx \right) ds \le \int_\Omega \phi_1(x) \, dx + e^{-r_u\lambda_1^\alpha t} \int_\Omega u_0(x) \phi_1(x) \, dx. 
\end{equation}

\noindent{\rm (ii)} For any $\delta > 0$, define the interior subdomain $\Omega_\delta := \{x \in \Omega : \mathrm{dist}(x, \partial\Omega) \ge \delta\}$ and $M_\delta := \inf_{x \in \Omega_\delta} \phi_1(x) > 0$. Then
\begin{equation}
\int_0^T \int_{\Omega_\delta} v(x,s)^2 \, dx \, ds \le \frac{1}{M_\delta} \left( e^{r_u\lambda_1^\alpha T}|\Omega|^{1/2} + \|u_0\|_2 \right). 
\end{equation}
\begin{proof}
    Since $-1 \le u(x,t) \le e^{-F t}(1- \|u_0\|_\infty)$ from Theorem \ref{bound_gs}, we have
\begin{equation}
   \int^t_0 e^{-r_u\flp (t-s)}(u(\cdot,s)+1)v(\cdot,s)^2 ds \le e^{-r_u\flp t}u_0(x) + 1, ~~\forall t\in[0,T],~x\in\Omega.\label{29}
\end{equation}
Multiplying both sides of \eqref{29} by $\phi_1(x)$ and integrating over $x \in \Omega$, we apply Fubini's theorem to obtain:
\[
\int_0^t \int_\Omega \phi_1(x) (e^{-r_u\flp (t-s)}v(\cdot,s)^2)(x) \, dx \, ds \le \int_\Omega \phi_1(x) \, dx + \int_\Omega \phi_1(x) (e^{-r_u\flp t}u_0)(x) \, dx.
\]
Since $e^{-r_u\flp t}$ is self-adjoint and $e^{-r_u\flp t}\phi_1 = e^{-r_u\lambda_1^\alpha t} \phi_1$, the left-hand side simplifies to:
\[
\int_0^t e^{-r_u\lambda_1^\alpha(t-s)} \left(\int_\Omega v(x,s)^2 \phi_1(x) \, dx \right) ds.
\]
The second term on the right-hand side is treated similarly:
\[
\int_\Omega \phi_1 e^{-r_u\flp t}u_0 \, dx = \int_\Omega u_0 e^{-r_u\flp t}\phi_1 \, dx = e^{-r_u\lambda_1^\alpha t} \int_\Omega u_0 \phi_1 \, dx.
\]
This yields (i). In particular, taking $t = T$ and using the fact that $e^{-\lambda_1^\alpha(T-s)} \ge e^{-\lambda_1^\alpha T}$ for $s \in [0,T]$, we obtain
\begin{equation}
\int_0^T \int_\Omega v(x,s)^2 \phi_1(x) \, dx \, ds \le e^{r_u\lambda_1^\alpha T} \int_\Omega \phi_1(x) \, dx + \int_\Omega u_0(x) \phi_1(x) \, dx.\label{30}
\end{equation}

Finally, since \(\phi_1\in C(\overline{\Omega})\) and \(\Omega_\delta\) is a compact subset of \(\Omega\), the minimum of \(\phi_1\) on \(\Omega_\delta\) is attained: \(M_\delta=\min_{\Omega_\delta}\phi_1\). Moreover, \(\phi_1>0\) in \(\Omega\) by the strong maximum principle, hence \(M_\delta>0\) and
\[
\int_{\Omega_\delta} v^2 \, dx \le \frac{1}{M_\delta} \int_{\Omega_\delta} v^2 \phi_1 \, dx \le \frac{1}{M_\delta} \int_\Omega v^2 \phi_1 \, dx.
\]
Since $\int_\Omega \phi_1(x) dx \le |\Omega|^{1/2}$ by the Cauchy–Schwarz inequality, integrating over $s \in [0,T]$ and applying the estimate from \eqref{30} yields (ii).
\end{proof}
\end{corollary}
\begin{remark}
    We should note that the prior estimation \eqref{35} can be extended from $\|v(t)\|_{L^2(\Omega_\delta)}$ to $\|v(t)\|_{L^2(\Omega)}$ when Neumann or periodic boundary conditions are adopted.
\end{remark}

From Corollary \ref{wellposedness_gs}, for any interior subdomain $\Omega_\delta$ there exists $t^*\in(0,T)$ such that
\[\|v(t^*)\|^2_{L^2(\Omega_\delta)}\le \frac{1}{T}\int^T_0 \int_{\Omega_\delta} v(x,s)^2dxds\le\frac{1}{T M_\delta} \left( e^{r_u\lambda_1^\alpha T}|\Omega|^{1/2} + \|u_0\|_2 \right).\]
Since $v\in C^1((0,T),C_0(\Omega))$, let $K_T = \sup_{t\in(0,T)}\|\partial_t v(t)\|_\infty$, then we have 
\begin{equation}
v(\cdot,t)=v(\cdot,t^*)+\int_{t^*}^t \partial_s v(\cdot,s)\,ds
\qquad \text{in } C_0(\Omega).
\end{equation}
Taking the $L^2(\Omega_\delta)$ norm and applying Minkowski's inequality yield
\begin{equation}
\begin{aligned}
    \|v(t)\|_{L^2(\Omega_\delta)}
   & \le
\|v(t^*)\|_{L^2(\Omega_\delta)}
+\int_{t^*}^t \|\partial_s v(s)\|_{L^2(\Omega_\delta)}\,ds\\
&\le
\|v(t^*)\|_{L^2(\Omega_\delta)}
+
|\Omega_\delta|^{1/2}
\int_{t^*}^t \|\partial_t v(s)\|_{L^\infty(\Omega)}\,ds.
\end{aligned}
\end{equation}
Consequently, for all $t\in(0,T)$,
\begin{equation}
\|v(t)\|_{L^2(\Omega_\delta)}
\le
\frac{1}{\sqrt{T M_\delta}} \left( e^{r_u\lambda_1^\alpha T}|\Omega|^{1/2} + \|u_0\|_2 \right)^{1/2}
+ |\Omega_\delta|^{1/2} T\,K_T.\label{35}
\end{equation}

\subsection{Numerical simulations}

For the Gray-Scott system \eqref{gs}, we apply the same numerical method that is used in Section~4. We set $L = 1$ and the parameters are chosen as $r_u = 2 r_v = 1e-6$. The different feeding and killing rates, $(F,\kappa)$, used for generating different pattern formations are $(0.026,0.063)$ and $(0.03,0.058)$.  The initial condition $(u_0,v_0)$ is chosen as a localized perturbation centered in the domain:
it is zero everywhere except in a small $D_{GS}:=\{(x,y):|x-0.5|=|y-0.5|\le0.04\}$ region where it takes the value $(0.5,0.25)$, i.e., 
\[(u_0(x,y),v_0(x,y)) = \left\{\begin{aligned}
    & (0.5,0.25),  ~~(x,y)\in D_{GS},\\
    & (0,0), ~~(x,y)\in \Omega\backslash D_{GS}
\end{aligned}\right.\]
where the discontinuity is regularized by convolution with a mollifier, ensuring $u_0,v_0\in C_0(\Omega)$. 

Figure \ref{pattern_gs_1} and Figure \ref{pattern_gs_2} first examine the case $\alpha=\beta$ for two different sets of parameters to show how the fractional order affects the resulting patterns. As $\beta$ increases, the system exhibits more pronounced long-range propagation, leading to the formation of organized, large-scale structures such as rings and spots. Conversely, for smaller $\beta$, the solution remains localized and retains more high-frequency components, leading to more fine-scale structures. This scale dependence comes from spectral decay.  Because the $k$-th eigenmode decays as $e^{-r_u\lambda_k^{\beta}t}$ and $e^{-r_v\lambda_k^{\beta}t}$, a lower value of $\beta$ attenuates the damping of high-frequency components. As a result, high-frequency eigenmodes persist longer at smaller fractional orders, leading to more intricate spatial patterns and weaker smoothing at small scales than in the classical case.

\begin{figure}[htbp]
\centering
\begin{minipage}[t]{0.24\textwidth}
\centering
\includegraphics[width=\textwidth]{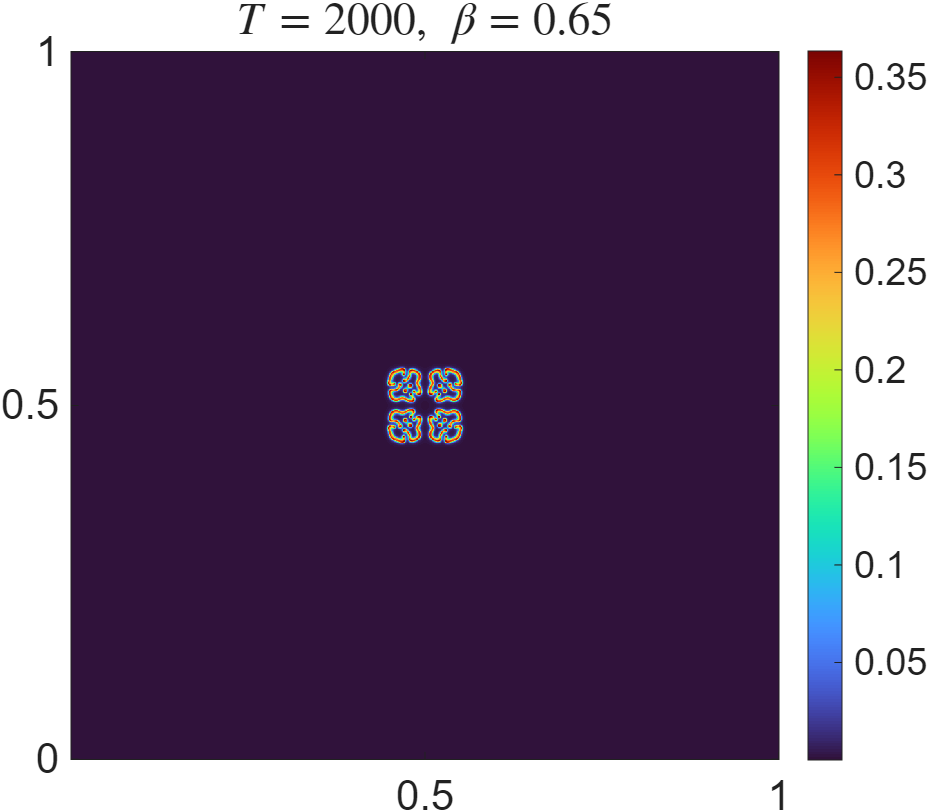}
\end{minipage}
\hfill 
\begin{minipage}[t]{0.24\textwidth}
\centering
\includegraphics[width=\textwidth]{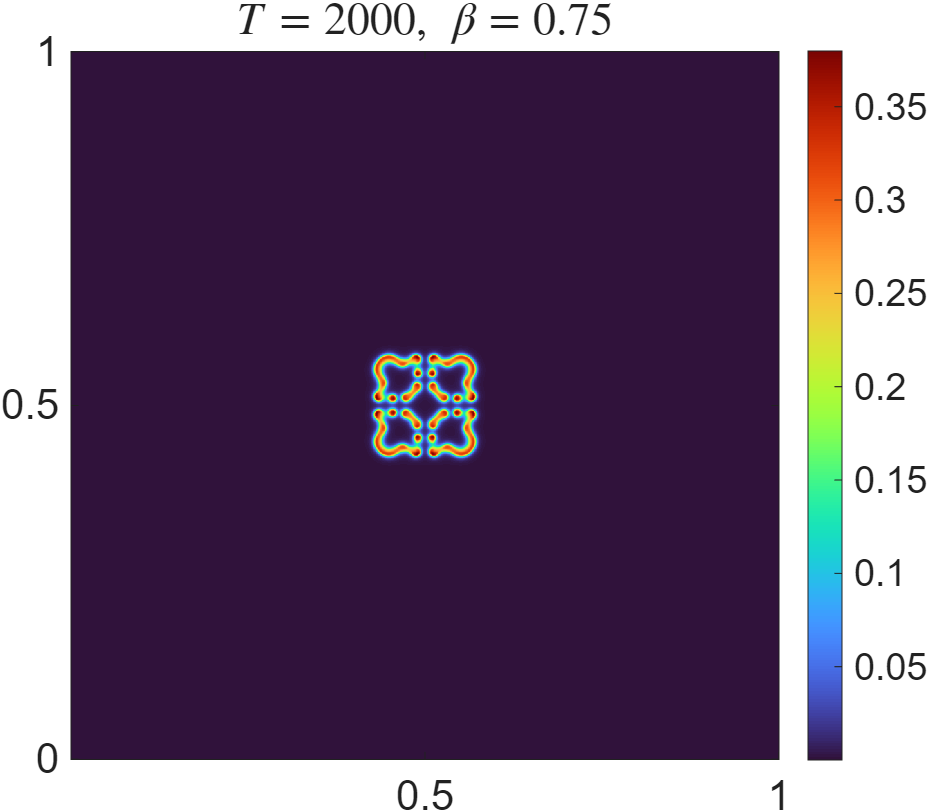}
\end{minipage}
\hfill
\begin{minipage}[t]{0.24\textwidth}
\centering
\includegraphics[width=\textwidth]{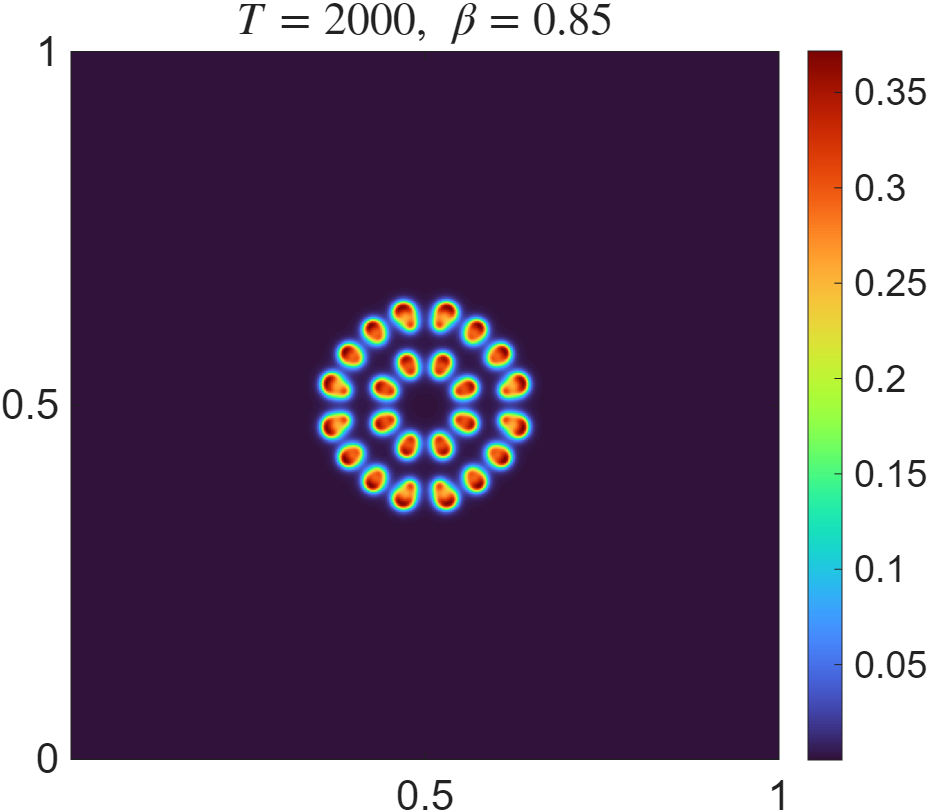}
\end{minipage}
\hfill
\begin{minipage}[t]{0.24\textwidth}
\centering
\includegraphics[width=\textwidth]{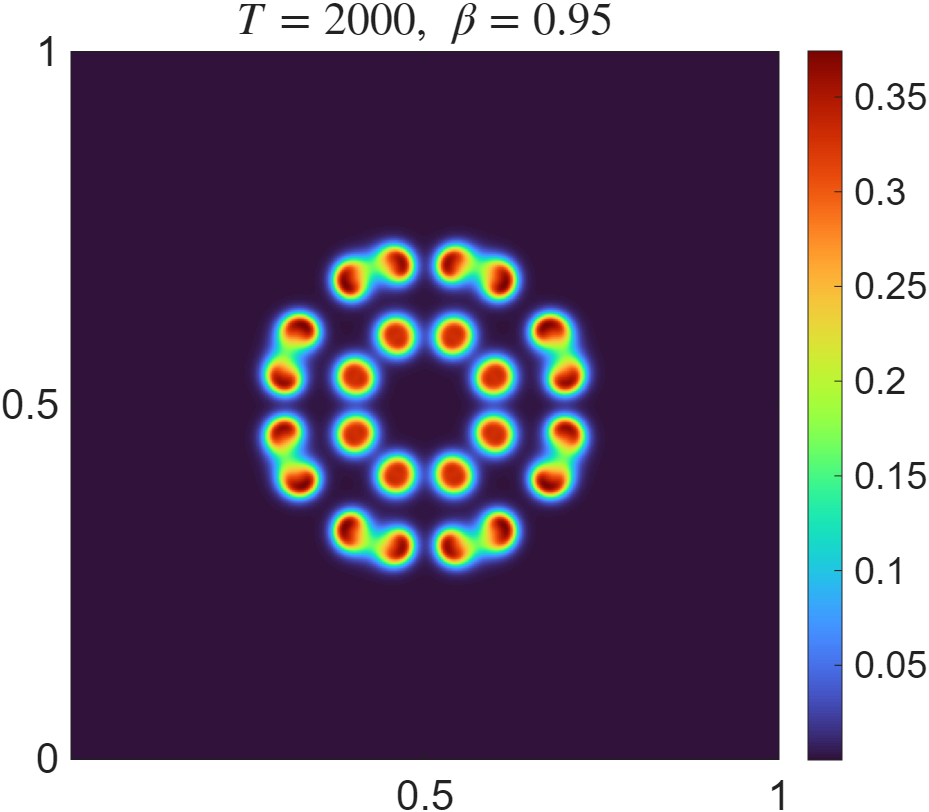}
\end{minipage}
\vspace{0.3cm} 

\begin{minipage}[t]{0.24\textwidth}
\centering
\includegraphics[width=\textwidth]{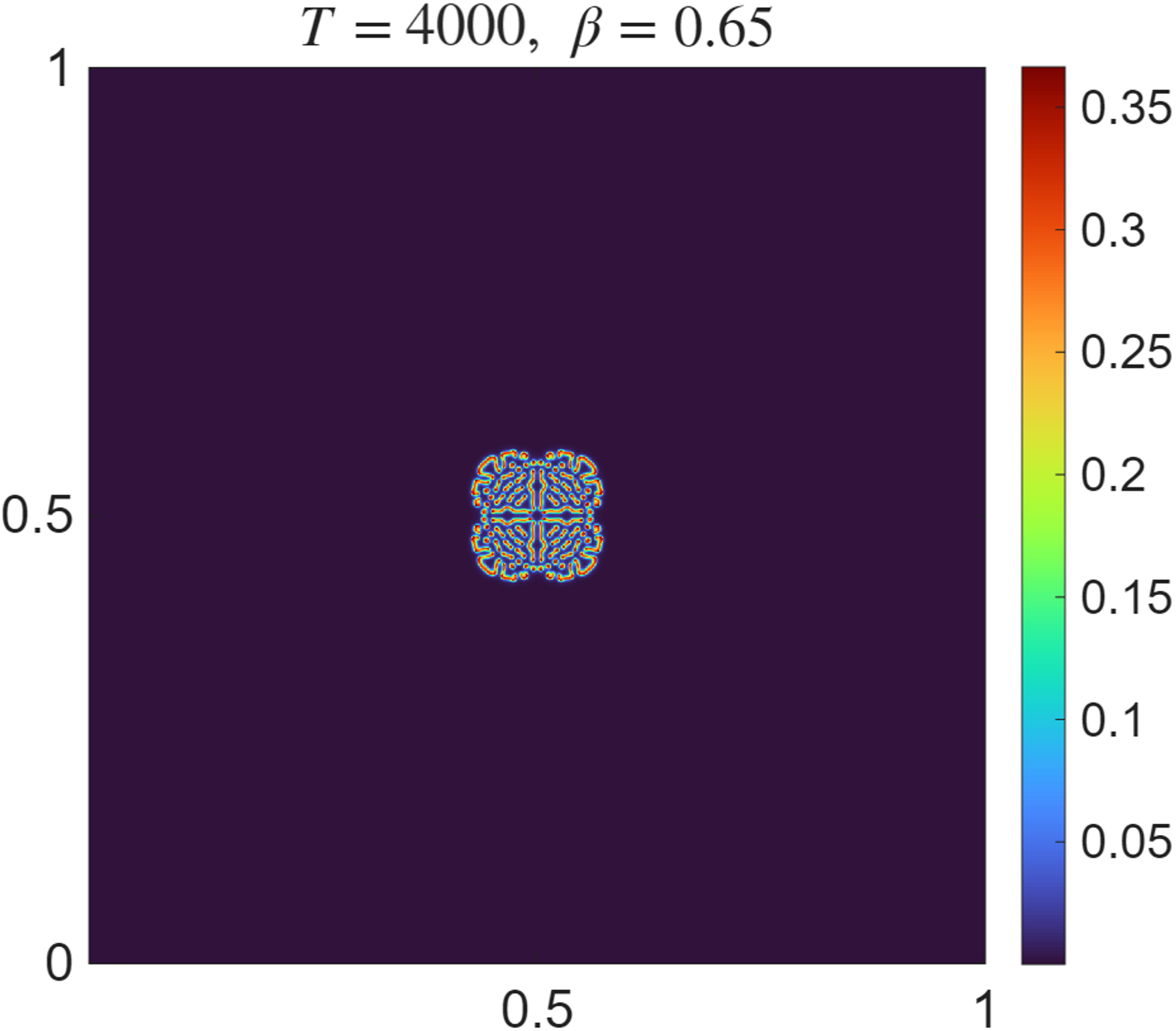}
\end{minipage}
\hfill
\begin{minipage}[t]{0.24\textwidth}
\centering
\includegraphics[width=\textwidth]{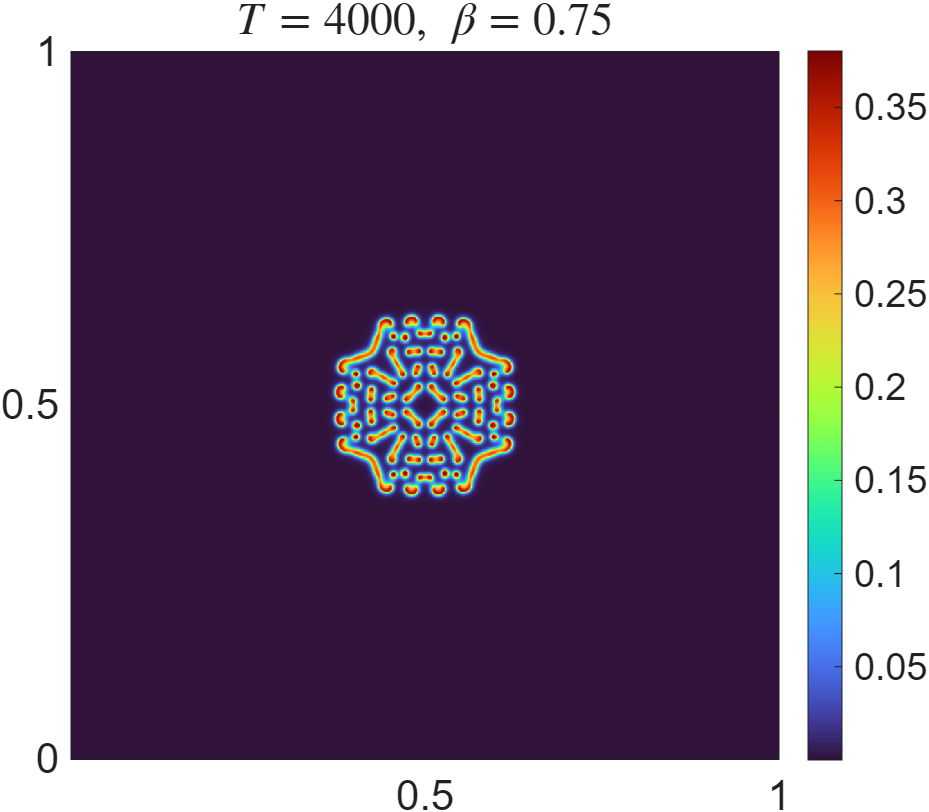}
\end{minipage}
\hfill
\begin{minipage}[t]{0.24\textwidth}
\centering
\includegraphics[width=\textwidth]{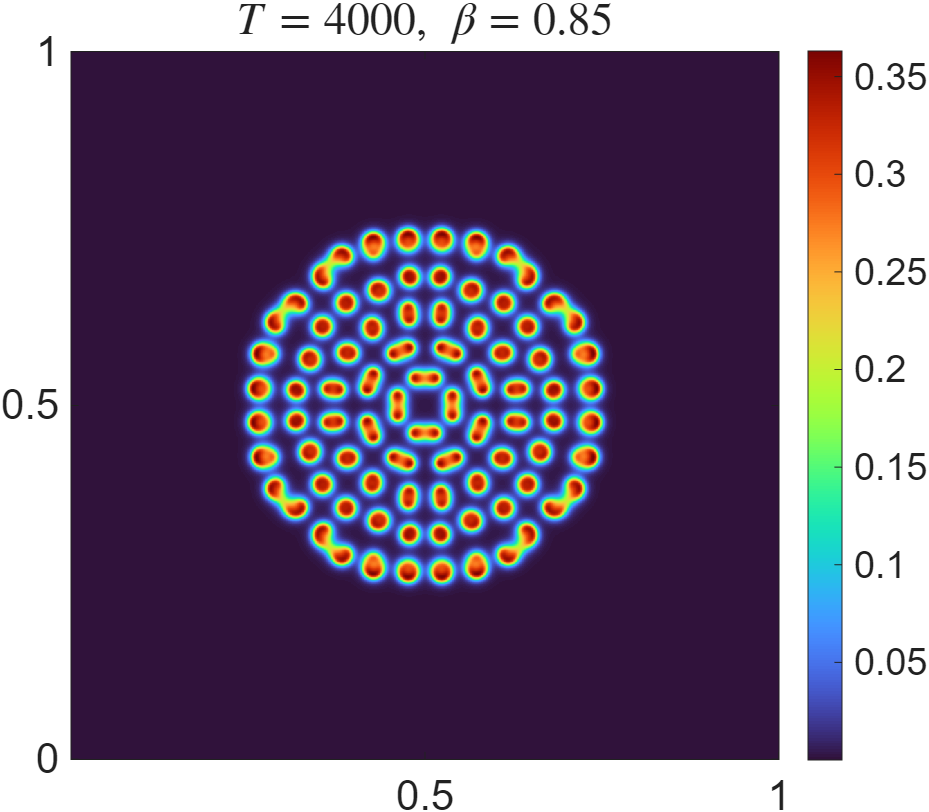}
\end{minipage}
\hfill
\begin{minipage}[t]{0.24\textwidth}
\centering
\includegraphics[width=\textwidth]{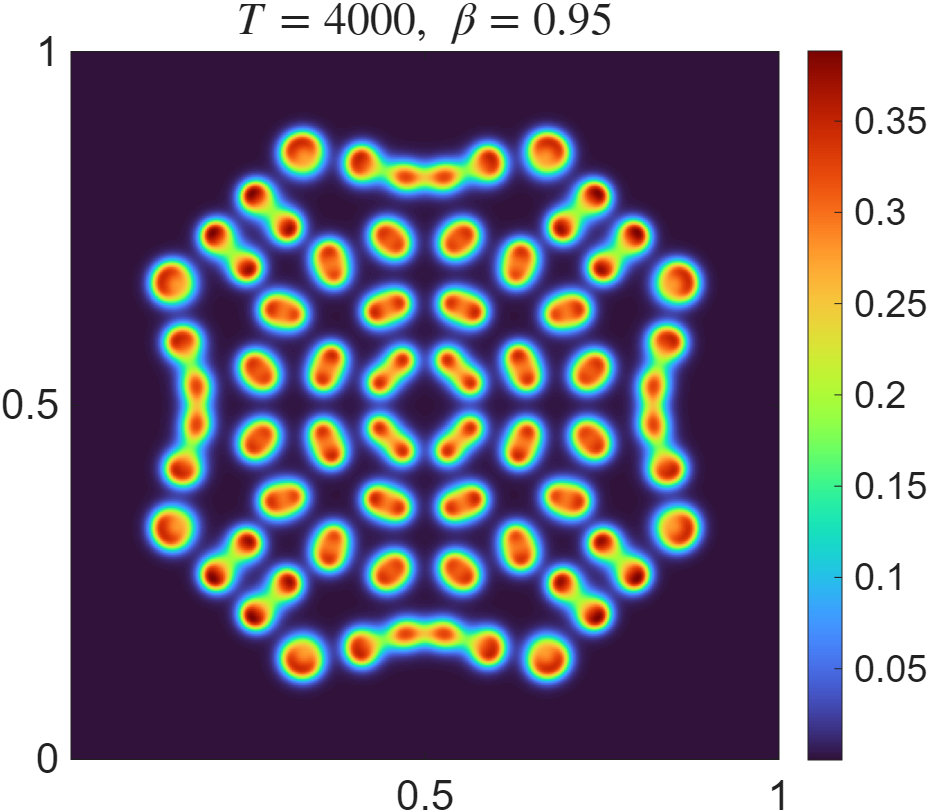}
\end{minipage}
\caption{Evolution of pattern formation in the fractional Gray-Scott model \eqref{gs}. The panels display the numerical solutions of the $v$-component for varying fractional orders $\beta \in \{0.65, 0.75, 0.85, 0.95\}$ at $T=2000$ (top) and $T=4000$ (bottom), with parameters $F = 0.026$ and $\kappa = 0.061$.}
\label{pattern_gs_1}
\end{figure}

\begin{figure}[htbp]
\centering
\begin{minipage}[t]{0.24\textwidth}
\centering
\includegraphics[width=\textwidth]{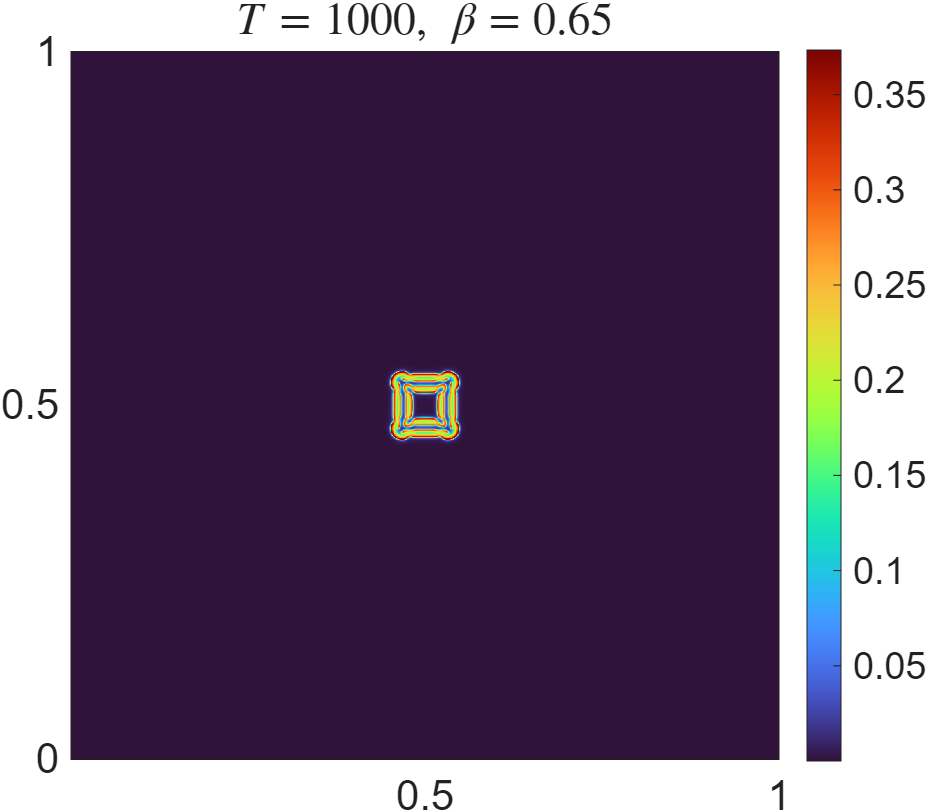}
\end{minipage}
\hfill 
\begin{minipage}[t]{0.24\textwidth}
\centering
\includegraphics[width=\textwidth]{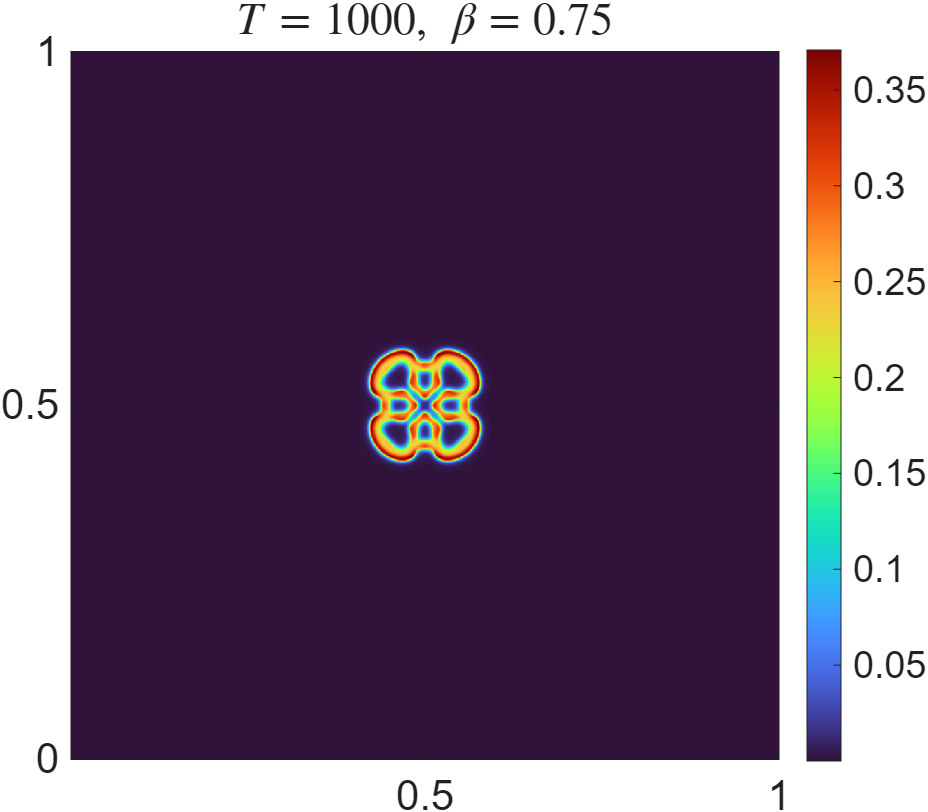}
\end{minipage}
\hfill
\begin{minipage}[t]{0.24\textwidth}
\centering
\includegraphics[width=\textwidth]{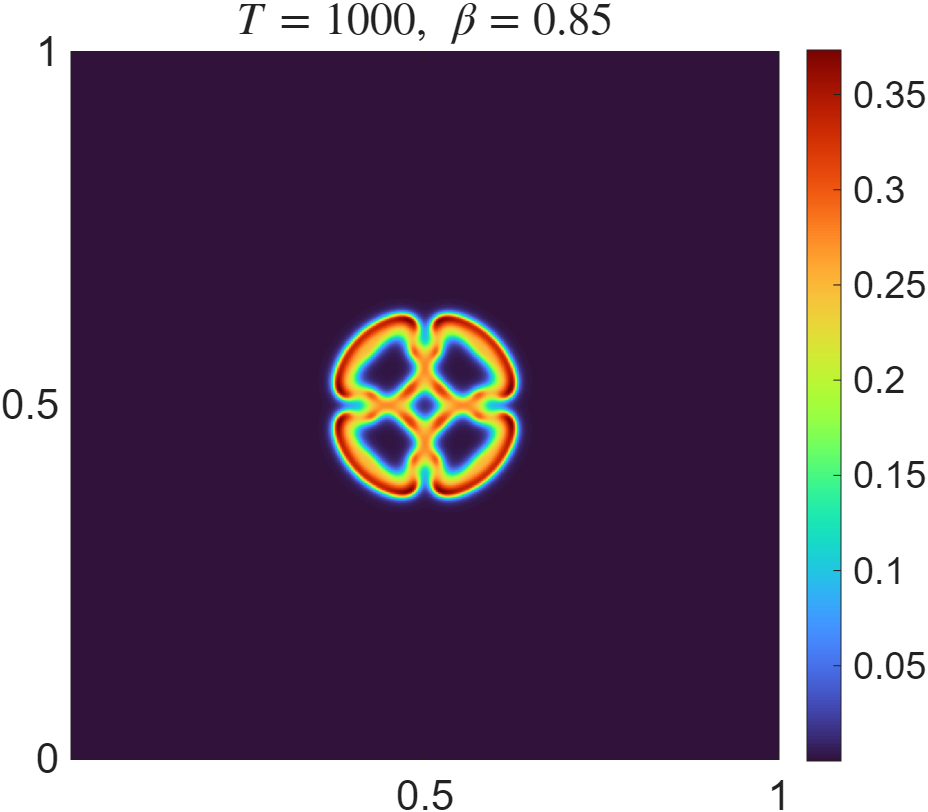}
\end{minipage}
\hfill
\begin{minipage}[t]{0.24\textwidth}
\centering
\includegraphics[width=\textwidth]{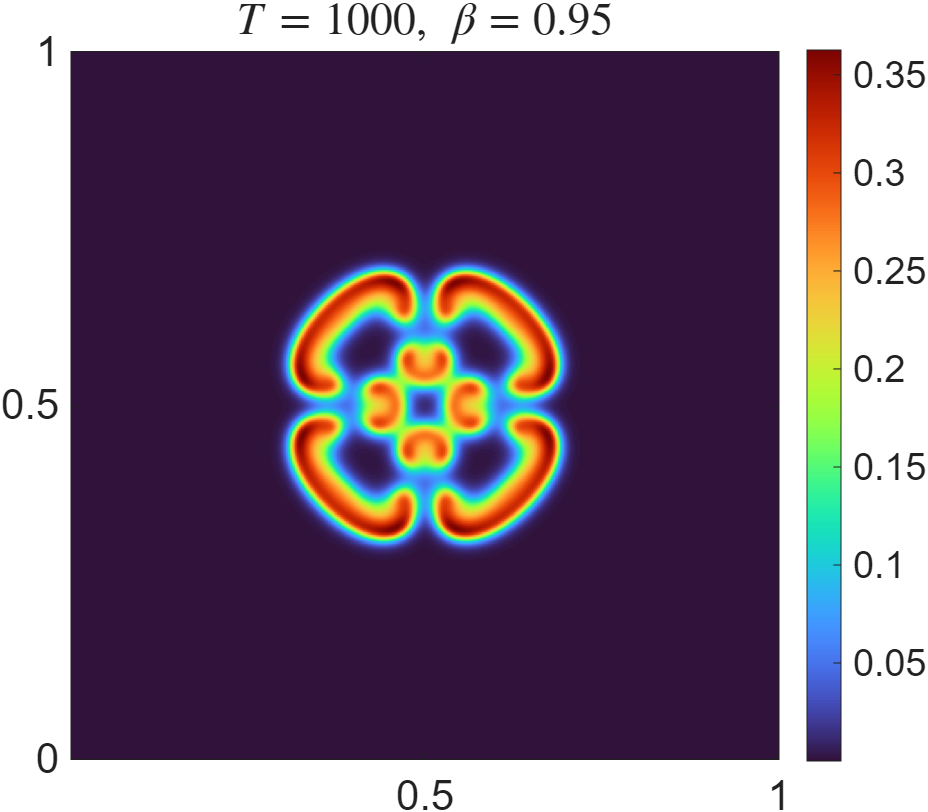}
\end{minipage}

\vspace{0.3cm} 

\begin{minipage}[t]{0.24\textwidth}
\centering
\includegraphics[width=\textwidth]{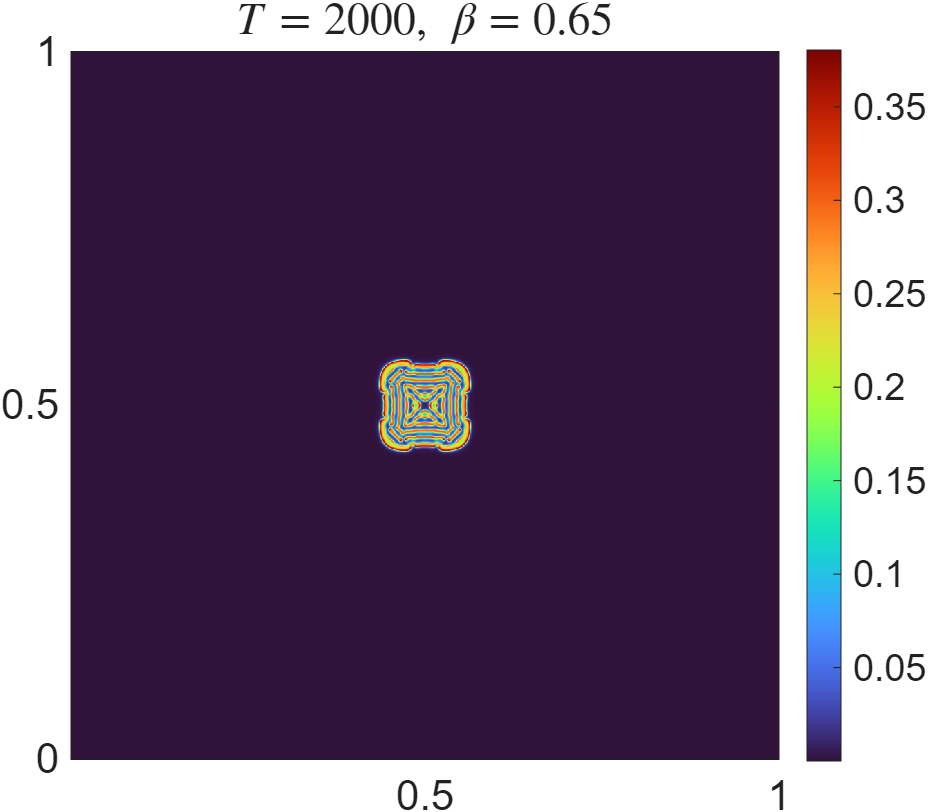}
\end{minipage}
\hfill
\begin{minipage}[t]{0.24\textwidth}
\centering
\includegraphics[width=\textwidth]{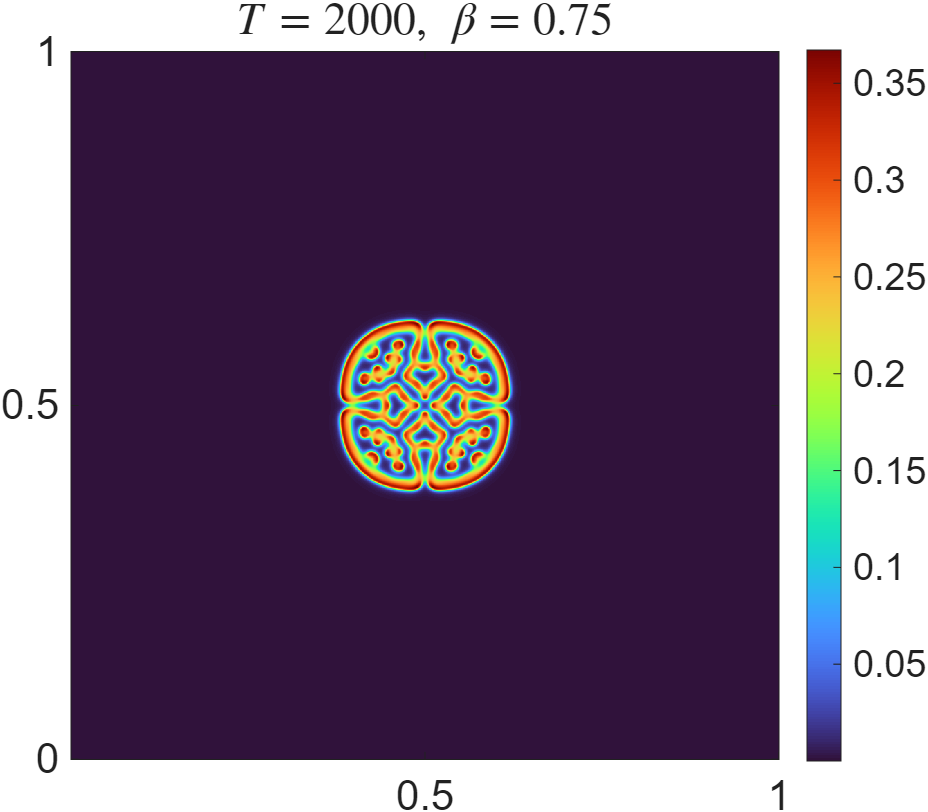}
\end{minipage}
\hfill
\begin{minipage}[t]{0.24\textwidth}
\centering
\includegraphics[width=\textwidth]{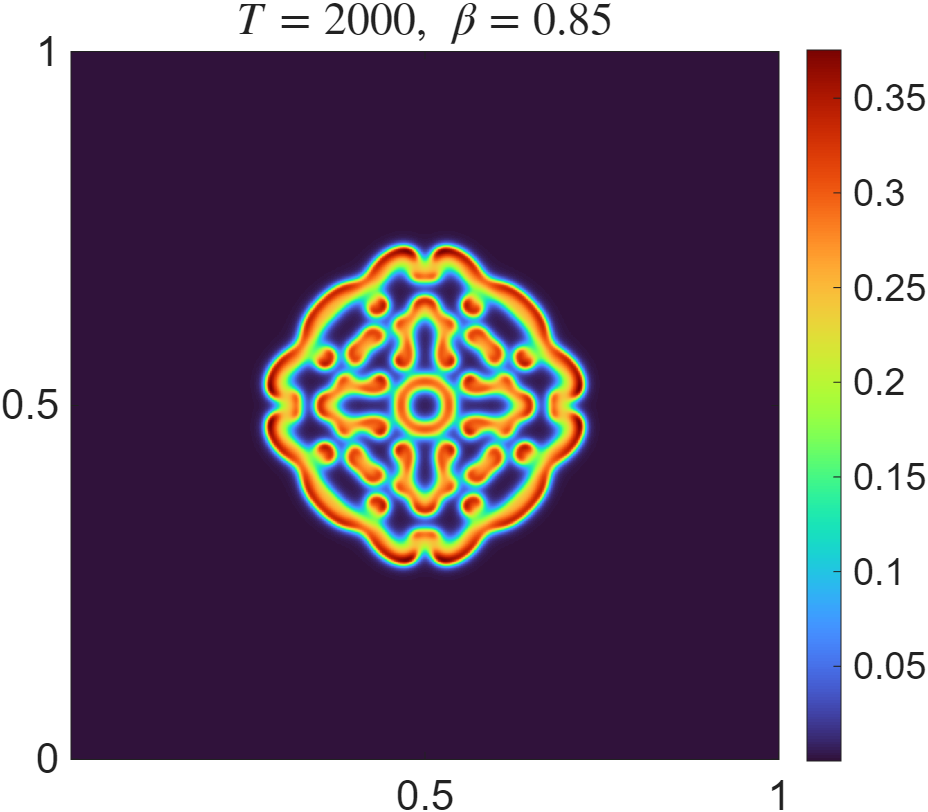}
\end{minipage}
\hfill
\begin{minipage}[t]{0.24\textwidth}
\centering
\includegraphics[width=\textwidth]{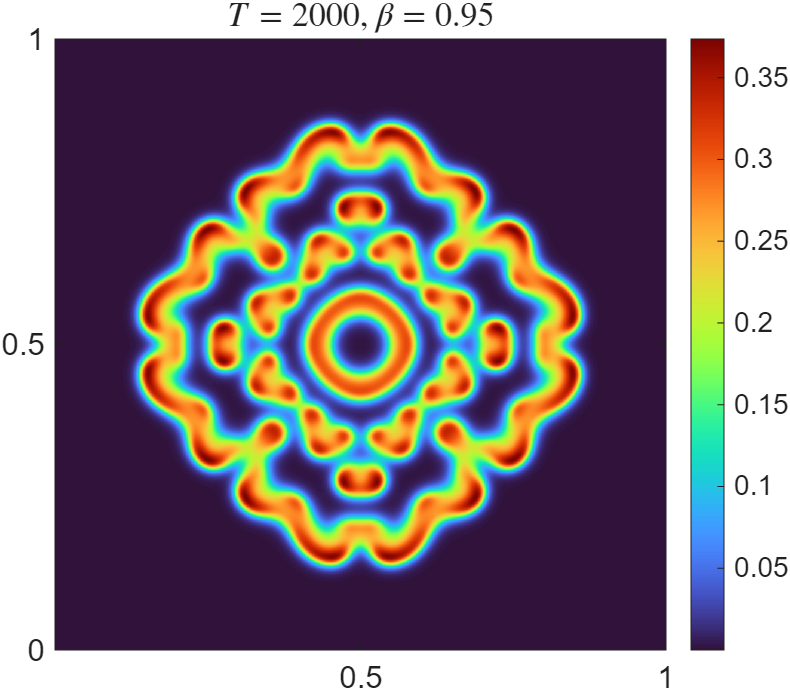}
\end{minipage}

\caption{Evolution of pattern formation in the fractional Gray-Scott model \eqref{gs}. The panels display the numerical solutions of the $v$-component for varying fractional orders $\beta \in \{0.65, 0.75, 0.85, 0.95\}$ at $T=1000$ (top) and $T=2000$ (bottom), with parameters $F = 0.03$ and $\kappa = 0.058$.}
\label{pattern_gs_2}
\end{figure}

\begin{figure}[htbp]
  \centering
\begin{minipage}[t]{0.25\textwidth}
    \includegraphics[scale=0.27]{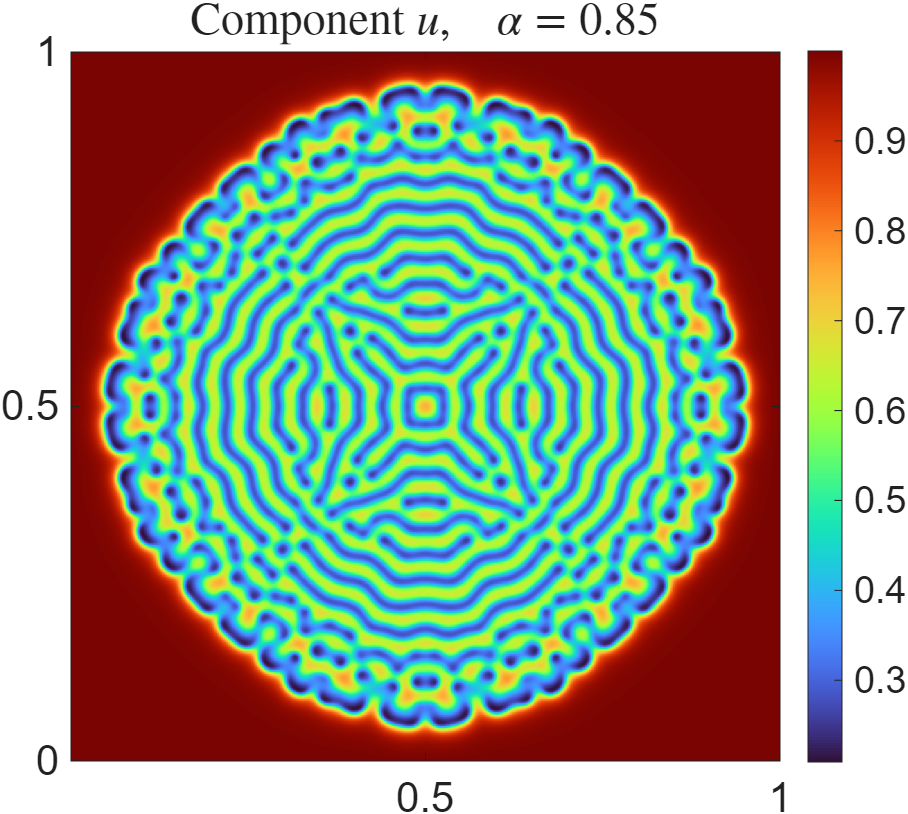}
\end{minipage}
  \centering
\begin{minipage}[t]{0.25\textwidth}
    \includegraphics[scale=0.27]{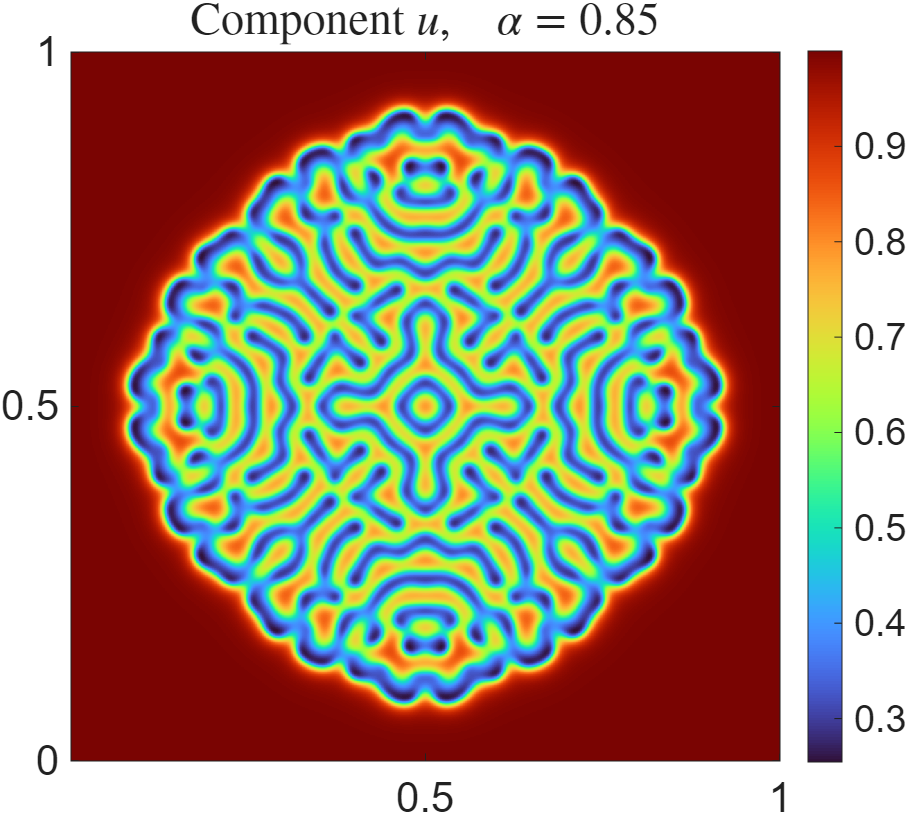}
\end{minipage}
  \centering
\begin{minipage}[t]{0.25\textwidth}
    \includegraphics[scale=0.27]{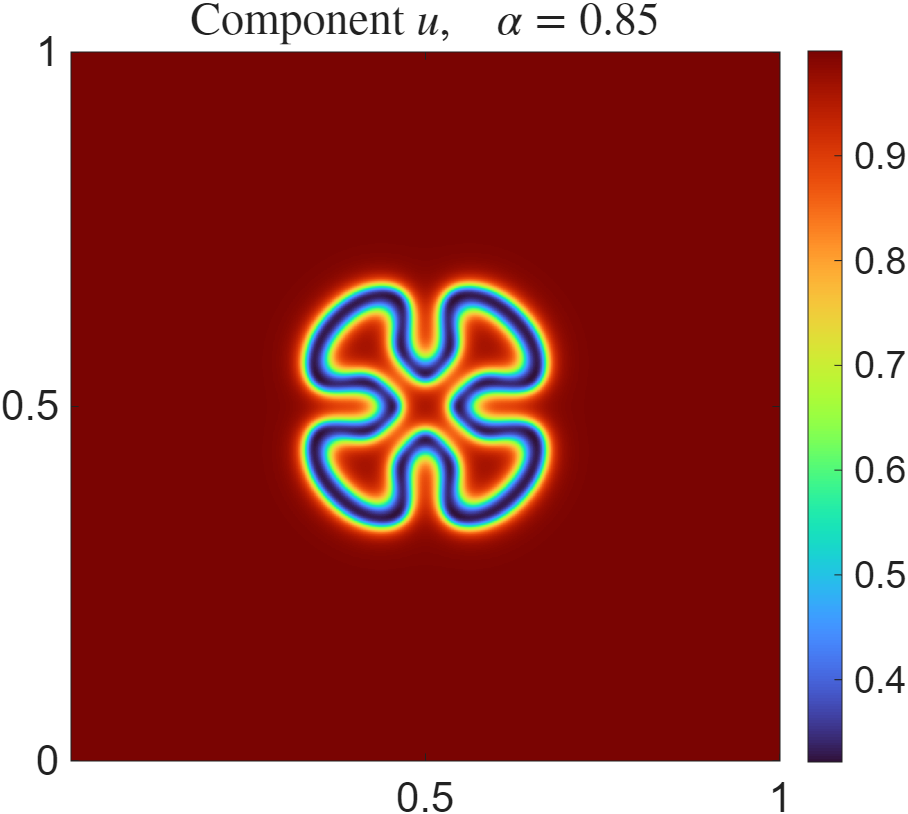}
\end{minipage}
  \centering
\begin{minipage}[t]{0.25\textwidth}
    \includegraphics[scale=0.27]{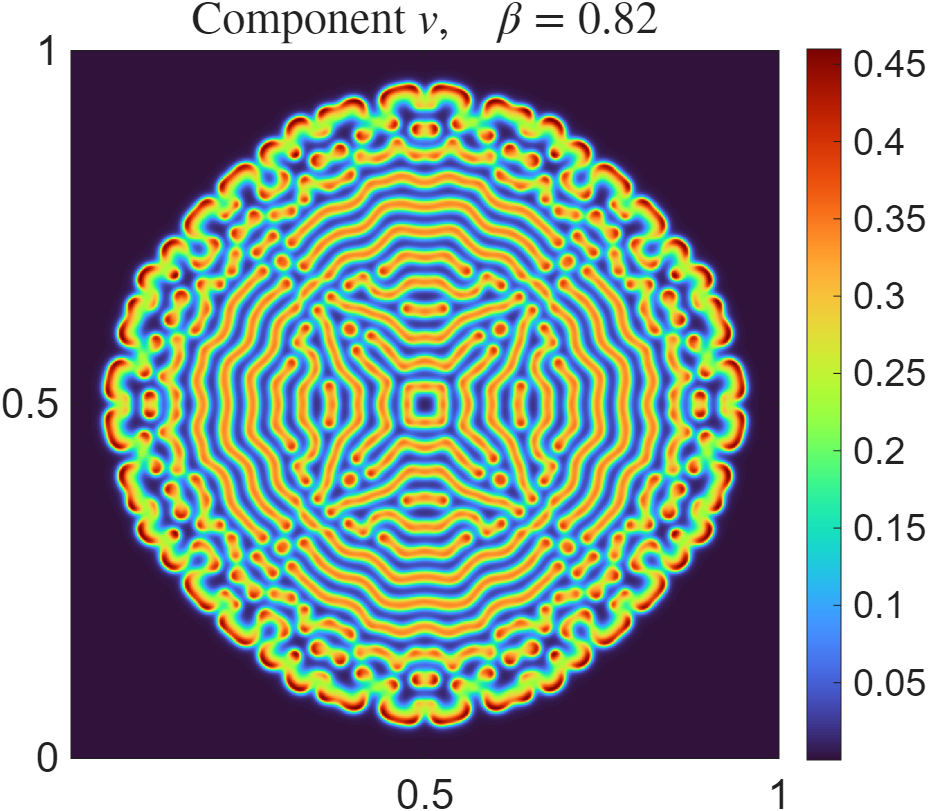}
\end{minipage}
  \centering
\begin{minipage}[t]{0.25\textwidth}
    \includegraphics[scale=0.27]{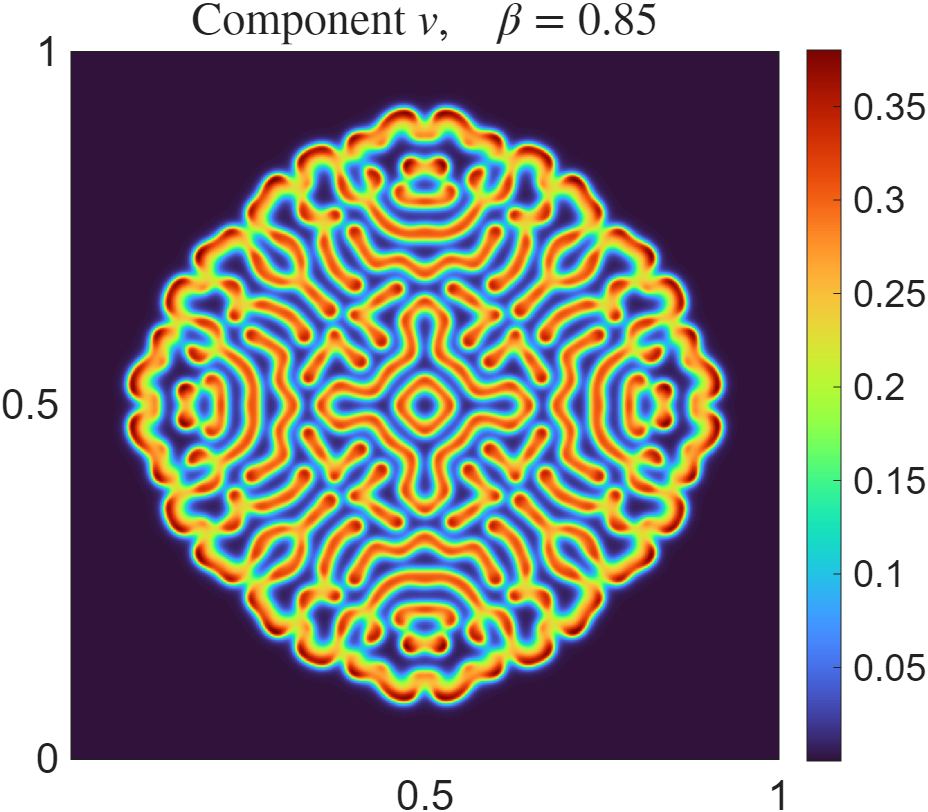}
\end{minipage}
  \centering
\begin{minipage}[t]{0.25\textwidth}
    \includegraphics[scale=0.27]{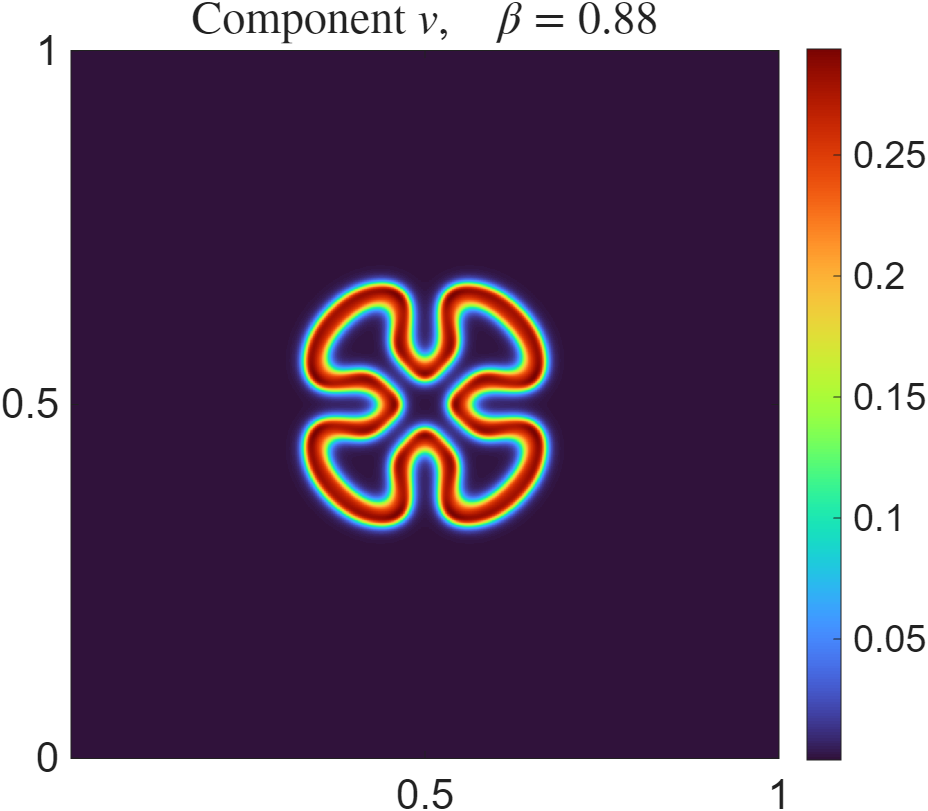}
\end{minipage}
\caption{Comparison of pattern formations for different $\beta \in \{0.82, 0.85, 0.88\}$ at $T=4000$. The numerical solutions for $u$ (top) and $v$ (bottom) with $F = 0.03$ and $\kappa=0.058$.}
\label{compare_gs_alpha_beta}
\end{figure}

Figure \ref{compare_gs_alpha_beta} compares the long-time spatial profiles produced by different fractional diffusion orders in the coupled system. Keeping \(\alpha\) fixed and increasing \(\beta\) strengthens the smoothing in the $v$-equation (each eigenmode decays like \(e^{-r_v\lambda_k^{\beta}t})\), so fine oscillations are damped more efficiently and the pattern scale becomes coarser: dense ring/labyrinth-like structures for smaller \(\beta\) change into fewer, more isolated ring-like spots for larger \(\beta\). This highlights that nontrivial patterns rely on differential diffusion between the two components, when the diffusion of $v$ becomes too strong relative to that of $u$, the coupling terms cannot sustain spatial modulation and the solution relaxes toward a spatially homogeneous equilibrium. This also explains why Gray–Scott studies often impose a clear diffusion contrast between the two components, and many classical parameter choices take \(D_u \gg D_v\) and $\alpha\ge\beta$ to promote robust localized structures and their replication-like dynamics in space \cite{wang2019fractional,yuan2024adaptive,zemskov2016diffusive}.

\section{Conclusions}
In this paper we present a unified analytical and computational framework for semilinear reaction--nonlocal diffusion equations driven by the spectral fractional Laplacian. We work in $C_0(\Omega)$ and use the semigroup $Q(t)=e^{-\epsilon^2(-\Delta)^\alpha t}$ together with the variation-of-constants formula to obtain well-posedness and instantaneous regularization. Moreover, under a quadratic-type control on the reaction term $\Nn$, we derive an $L^2$ a priori estimate that guarantees global well-posedness.

A key technical issue is the treatment of inhomogeneous Dirichlet conditions for spectral operators. Because prescribing a fixed nonzero trace yields an admissible class that is not a linear space, we review the inhomogeneous spectral fractional Laplacian and apply a lifting approach (shift) to reduce the problem to homogeneous boundary conditions. 

For the bistable RNDE \eqref{bistable}, we established an energy dissipation identity that yields monotone energy decay, a coercivity bound, and uniform-in-time control of weak solutions in $\mathbb H^{\alpha}(\Omega)$, combined with the compact embedding $\mathbb H^{\alpha}(\Omega)\hookrightarrow L^2(\Omega)$, it implies a nonempty $\omega$-limit set and identifies every $\omega$-limit point as a weak stationary solution. A fractional weak maximum principle further gives an invariant-range property between the two stable equilibria, and we also establish the corresponding weak maximum principle for the inhomogeneous operator $\flpg$. For the nonlocal Gray--Scott system \eqref{ori_gs}, we obtain positivity preservation for the components and an explicit global $L^\infty$ invariant region that prevents blow-up and yields global well-posedness. Moreover, an eigenfunction-weighted interior $L^2$ bound for $v$ tied to the first Dirichlet eigenpair quantifies how the $v$-component remains controlled away from the boundary.

On the computational side, the DST-I pseudospectral discretization exploits the Dirichlet eigenstructure and, when paired with ETDRK4, treats the stiff fractional diffusion through the semigroup while keeping the nonlinear part explicit. The resulting solver remains accurate over long horizons and reveals qualitative transitions as the fractional orders vary, consistent with the invariant/energy bounds and clarifying the modeling impact of using the spectral operator on bounded domains.

Future work will focus on extending the dynamical and analytical understanding of the nonlocal Gray--Scott system, as well as on developing efficient discretizations for the Dirichlet fractional Laplacian $\flp_D$. For the Gray--Scott model, we are currently investigating the steady-state structure and its bifurcations, with the aim of explaining pattern formation through the study of homogeneous equilibria stability and Turing-type instabilities associated with nonlocal diffusion, and of supporting the analysis through numerical continuation of stationary branches. We also plan to extend the interior weighted energy estimates for $(u,v)$ to the entire domain $\Omega$, aiming to establish global well-posedness and regularity in $L^2(\Omega)$ by imposing additional growth conditions on the reaction terms. In addition, we will study RNDEs driven by the Dirichlet (integral) fractional Laplacian $\flp_D$. By Proposition~\ref{compare_2_Laplacians}, for convex domains the eigenvalues of $\flp_D$ satisfy the two-sided bound \(\frac{1}{2}\lambda_k^\alpha \le \zeta_k \le \lambda_k^\alpha.\) This bound provides a concrete link between $\flp_D$ and the spectral fractional Laplacian, and it motivates transferring efficient spectral ideas developed here to the integral-operator setting, in particular the design of fast discretizations for $\flp_D$ that preserve the accuracy and long-time stability observed for DST-based spectral solvers.

\appendix
\section{Numerical method}
We briefly describe the sine pseudospectral discretization (DST--I) in space and the ETDRK4 time integrator used for the RNDE system defined in $\Omega=(0,L)^2$ and homogeneous Dirichlet boundary conditions are imposed.
\subsection{Spatial discretization (DST-I).}
Let the interior grid points be $x_i=\frac{iL}{N_x+1} (i=1,\dots,N_x)$ and $y_j=\frac{jL}{N_y+1} (j=1,\dots,N_y)$. With Dirichlet boundary conditions we expand in the tensor-product sine basis
$(\phi_k(x)=\sin(k\pi x/L)), (\psi_m(y)=\sin(m\pi y/L))$.
Then the inverse DST-I reads
\[
\widehat U=S_x^\top U S_y,\qquad
U=\frac{2}{N_x+1}\frac{2}{N_y+1} S_x\widehat US_y^\top.
\]
where $U = u_{ij}$ and $S_{x,ik}=\sin(\frac{\pi ik}{N_x+1}), S_{y,jm}=\sin(\frac{\pi jm}{N_y+1})$.
The Laplacian diagonalizes:
\[
\lambda_{k m}=\Big(\tfrac{k\pi}{L}\Big)^2+\Big(\tfrac{m\pi}{L}\Big)^2,\qquad
\flp \ \Longleftrightarrow\ \lambda_{k m}^{\alpha}.
\]
Define the modal linear symbols
\[
\mathcal{L}_{k m}=-\epsilon^2\,\lambda_{k m}^{\alpha}.
\]
Then we conclude the semi-discrete system of \eqref{ori_eq} as 
\[
\frac{d}{dt}\widehat U_{k m}= \mathcal{L}_{k m}\widehat U_{k m}+\widehat{\mathcal N(U)}_{k m},
\]
where nonlinear terms are evaluated pointwise in physical space and then transformed.

\subsection{Time stepping (ETDRK4)}After a sine pseudospectral (DST--I) discretization of $\Omega=(0,L)^2$ with homogeneous Dirichlet boundary conditions, 
the RNDE can be written as
\begin{equation}\label{eq:semi}
    \frac{d}{dt}\widehat U(t) = \mathcal{L}\,\widehat U(t) + \widehat{\Nn(U(t))}, \qquad \widehat U(0) = \widehat U_0,
\end{equation}
where $\widehat U=(\widehat u,\widehat v)^\top$.
The nonlinear reaction term $\widehat{\Nn(U)}$ is obtained by transforming the pointwise reaction
$\Nn(u)$ from the physical grid to spectral space via the DST--I transform. 
The exact solution of \eqref{eq:semi} satisfies the variation--of--constants formula
\begin{equation}\label{eq:voc}
    \widehat U(t) = e^{t\mathcal{L}} \widehat U_0 + \int_{0}^{t} e^{(t-\tau)\mathcal{L}}\,\widehat{\Nn\big(U(\tau)}\big) d\tau .
\end{equation}
Exponential integrators approximate only the nonlinear part in \eqref{eq:voc}, 
while the stiff linear operator $\mathcal{L}$ is treated exactly through the matrix exponential.

\paragraph{Formulation of the ETDRK4 scheme.\cite{etdrk4_0}}
Let $t_n$ denote the current time and $\Delta t>0$ the time step.
Define the operator functions
\[
\mathcal{E} = e^{\Delta t \mathcal{L}}, \qquad \mathcal{E}_{\frac 12} = e^{\frac{\Delta t}{2} \mathcal{L}},
\]
and the auxiliary $\Phi$--functions
\[
\Phi_1(\Delta t \mathcal{L}) = \mathcal{L}^{-1}(\mathcal{E} - I), \qquad
\Phi_1^{\frac 12}(\Delta t \mathcal{L}) = \mathcal{L}^{-1}(\mathcal{E}_{\frac 12} - I),
\]
\[
\Phi_2(\Delta t \mathcal{L}) = \mathcal{L}^{-2}(\mathcal{E} - I - \Delta t \mathcal{L}), \qquad
\Phi_3(\Delta t \mathcal{L}) = \mathcal{L}^{-3}\!\left(\mathcal{E} - I - \Delta t \mathcal{L} - \tfrac{1}{2}(\Delta t \mathcal{L})^2\right),
\]
applied elementwise to each spectral mode. 
Given $\widehat U_n$ at time $t_n$, the ETDRK4 stages are defined by
\begin{align*}
a_n &= \mathcal{E}_{\frac 12} \widehat U_n + \Phi_1^{\frac 12}(\Delta t \mathcal{L}) \widehat{\Nn(U_n)}, \\
b_n &= \mathcal{E}_{\frac 12} \widehat U_n + \Phi_1^{\frac 12}(\Delta t \mathcal{L}) \widehat{\Nn(a_n)}, \\
c_n &= \mathcal{E}_{\frac 12} a_n + \Phi_1^{\frac 12}(\Delta t \mathcal{L})\big(2\widehat{\Nn(b_n)} - \widehat{\Nn(U_n)}\big).
\end{align*}
The solution is then advanced as
\begin{equation}\label{eq:etdrk4}
\begin{aligned}
\widehat U_{n+1} &= \mathcal{E} \widehat U_n 
+ \Phi_1(\Delta t \mathcal{L}) \widehat{\Nn(U_n)}
+ 2 \Phi_2(\Delta t \mathcal{L})\big(\widehat{\Nn(a_n)} + \widehat{\Nn(b_n)}\big)
+ \Phi_3(\Delta t \mathcal{L}) \widehat{\Nn(c_n)} .
\end{aligned}
\end{equation}

All operations in \eqref{eq:etdrk4} are performed mode-by-mode (and componentwise for vector-valued systems) in spectral space. 
In practice, the functions $\Phi_j(\Delta t \mathcal{L})$ are computed using stabilized contour averages or the equivalent $\varphi$--function representation
$\Phi_j = (\Delta t)^j \varphi_j(\Delta t \mathcal{L})$, 
which avoids cancellation for small $\|\Delta t \mathcal{L}\|_{\mathrm{op}}$.
The resulting ETDRK4 scheme is explicit, fourth--order accurate in time, and treats the stiff fractional diffusion operator exactly through the exponential while approximating only the nonlinear reaction term, see \cite{kassam2005fourth,schmelzer2006evaluating} for more details.

For reproducibility, the bistable simulations in Section~4 use time step $\Delta t=10^{-2}$ and the Gray--Scott computations in Section~5 use time step $\Delta t=4\times10^{-1}$. Both simulations use $16$-point complex contour averages. For each spectral mode we set
\[
r_j=\exp\!\left(i\pi(j-\tfrac12)/16\right),\qquad
Z_j=\Delta t\,\mathcal L+r_j,\qquad j=1,\ldots,16,
\]
and compute the ETDRK4 coefficients $Q,f_1,f_2,f_3$ by the standard stabilized averages over the values $Z_j$ \cite{kassam2005fourth,schmelzer2006evaluating}. In all inhomogeneous-boundary simulations, the DST-I is applied only to the shifted zero-trace variable. Thus, for a constant boundary value $g$ we evolve $w=u-g$ and add $g$ back only when plotting the original variable.

\bibliographystyle{IEEEtran}
\bibliography{ref}
\end{document}